\documentclass[a4paper,11pt,reqno]{amsart}
\usepackage{amsmath}
\usepackage{amssymb}
\usepackage{amsfonts}
\usepackage{graphicx}
\usepackage{mathtools}
\usepackage[colorlinks]{hyperref}

\renewcommand\eqref[1]{(\ref{#1})} 
\graphicspath{ {images/} }
\usepackage{thmtools}

\title[Isoperimetric inequlities and spectral consequences]{Isoperimetric Inequalities and Spectral Consequences in warped product manifolds}

\author[A. Banerjee]{Avas Banerjee}
	\address[A. Banerjee]{Theoretical Statistics and Mathematics Unit, Indian Statistical Institute, Delhi Center, S.J. Sansanwal Marg, New Delhi, Delhi 110016, India
	}
	\email{avas24r@isid.ac.in}


\date{\today}

\theoremstyle{plain}
\newtheorem{theorem}{Theorem}[section]
\newtheorem{proposition}{Proposition}[section]
\newtheorem{lemma}{Lemma}[section]
\newtheorem{corollary}{Corollary}[section]
\newtheorem{remark}{Remark}[section]
\newtheorem{definition}{Definition}[section]

\numberwithin{equation}{section} \allowdisplaybreaks

\usepackage[text={6.8in,10in},centering]{geometry}
\newcommand{\hn}{\mathbb{H}^n}

\newcommand{\dvg}{\: {\rm d}v_{g}}
\newcommand{\gradg}{\nabla_g}

\newcommand{\lapg}{\Delta_{g}}

\newcommand{\m}{\mathbb{M}^n}
\newcommand{\dx}{\:{\rm d}x}

\newcommand{\dt}{\:{\rm d}t}
\newcommand{\ds}{\:{\rm d}s}
\newcommand{\dr}{\:{\rm d}r}

\newcommand{\dsn}{\:{\rm d}\theta}
\newcommand{\sn}{\mathbb{S}^{n-1}}
\newcommand{\rn}{\mathbb{R}^n}
\newcommand{\uv}{\omega_n}

\begin{document}
\begin{abstract}
 
In this article, we investigate the \it centered isoperimetric inequality \rm on Cartan–Hadamard manifolds endowed with a warped product structure, namely, among all bounded measurable sets of finite perimeter and prescribed volume, the geodesic ball centered at the pole minimizes the perimeter. Exploiting the interplay between this inequality and the underlying warped product structure, we derive several necessary geometric conditions, some of which are closely related to and comparable with phenomena identified in the work of Simon Brendle [Publ. Math. Inst. Hautes Études Sci. 117  (2013)]. We also establish a sufficient condition ensuring the validity of the centered isoperimetric inequality in this setting. Furthermore, by introducing a suitable isoperimetric-type quotient, we obtain an improvement of the classical \it Cheeger inequality \rm for a broad class of manifolds. Finally, we derive a quantitative lower bound for the first nonzero Dirichlet eigenvalue of geodesic balls centered at the pole, valid for a certain class of Riemannian manifolds.   
 \end{abstract}

\maketitle

\section{Introduction}

The study of isoperimetric inequalities stands as one of the most venerable and central themes in both geometry and analysis. At its core, the classical isoperimetric problem seeks to determine the set of a given volume that minimizes its boundary area. In the context of Euclidean space $\mathbb{R}^n$, the solution is well-known: the unique minimizers are the round balls. However, as we transition from the flat Euclidean setting to the more complex landscape of Riemannian manifolds, the nature of this problem changes significantly, often becoming intricately linked to the underlying curvature and topological properties of the space.

One of the central conjectures in geometric analysis is the Cartan--Hadamard conjecture. Heuristically, it asserts that for a prescribed volume, any measurable set in a negatively curved space has a perimeter no smaller than that of a Euclidean ball of the same volume. A precise formulation is given below: 

\medskip

\textbf{Cartan-Hadamard conjecture} (Aubin \cite{Aubin}): Let $\m$ be an $n$-dimensional Cartan-Hadamard manifold, i.e., complete, simply-connected, and having everywhere non-positive sectional curvature. The conjecture says that the Euclidean isoperimetric inequality holds on $\m$: for every bounded, measurable set $\Omega\subset \m$ it holds that
\begin{equation}\label{CH-Cojecture-1}
    \text{Per}_{\m}(\Omega)\geq n|B(0,1)|^{\frac{1}{n}}[\text{Vol}_{\m}(\Omega)]^{\frac{n-1}{n}},
\end{equation}
where $|B(0,1)|$ is the Euclidean volume of the unit ball. Furthermore, equality holds for the set $\Omega$ if and only if there is a distance preserving map (upto a set of measure zero) 
\begin{equation*}
    \Phi: \Omega \to B_{R}(0),
\end{equation*}
where $B_{R}(0)$, the ball of radius $R$ around origin in $\rn$, carries the Euclidean metric and $\Omega$ is endowed with the induced Riemannian metric. Hence, no genuinely curved domain can realize equality unless it is metrically indistinguishable from the Euclidean ball.

\medskip

To date, the conjecture has been established for general Cartan–Hadamard manifolds up to dimension $4$: the two-, three-, and four-dimensional cases were proved in \cite{BR,Weil}, \cite{Kleiner}, and \cite{Croke}, respectively. The conjecture is also known to hold in hyperbolic space—namely, complete, simply connected manifolds with constant negative sectional curvature—in all dimensions. For a comprehensive survey of the Cartan–Hadamard conjecture and related developments, we refer to \cite{KK} and the references therein; see also \cite{Osserman} for a broader overview of isoperimetric problems

Beyond the scope of the Cartan–Hadamard conjecture, one may consider alternative isoperimetric problems that are intrinsic to the ambient manifold. Rather than comparing the perimeter of a set in a curved space with that of a Euclidean ball of the same volume, one compares perimeters entirely within the manifold itself. This leads to the notion of the \textit{centered isoperimetric inequality (CII)}.

\medskip

\textbf{Centered isoperimetric inequality (CII)}: A Riemannian manifold $\m$ possesing a pole $``0"$ is said to satisfy the CII if the following inequality holds:
\begin{equation}\label{Cenetered-2}
    \text{Per}_{\m}(B_r(0))\leq \text{Per}(\Omega),
\end{equation}
where $\Omega$ is any bounded, Borel measurable set in $\m$ and $B_r(0)$ is the geodesic ball centered at the pole of radius $r>0$ and have the same volume as $\Omega$.

\medskip

 The central question is whether such centered geodesic balls minimize perimeter among all sets of fixed volume, in close analogy with the minimizing property of Euclidean balls in the classical Cartan–Hadamard conjecture. So far, the centered isoperimetric inequality is known to hold only in the three model spaces: Euclidean space \cite{Ros}, hyperbolic space \cite{Schmidt}, and the sphere \cite{ps-ball}. The corresponding quantitative versions were proved in \cite{FMP}, \cite{ps-hn}, \cite{Fusco} for the Euclidean space, hyperbolic space and the sphere respectively.

The validity of the centered isoperimetric inequality is strongly related to the monotonicity property of the curvature of the manifold.
 Our first main theorem is regarding the sufficient condition of the centered isoperimetric inequality.  Nowadays, it is well understood that an isoperimetric inequality guarantees that
the decreasing rearrangement (see \eqref{centered symmetrization}) decreases the overall Dirichlet energy, and
this principle is commonly referred to as the P\'olya-Szeg\"o inequality. We refer to the survey papers \cite{Talenti} and references therein for the details about the P\'olya-Szeg\"o inequality. An $L^1$-type P\'olya--Szeg\"o inequality, when combined with some properties of the perimeter functional, yields the centered isoperimetric inequality.

\begin{theorem}\label{first}
Let $(\m, g)$ be an $n$-dimensional Riemannian manifold. 
Assume that the $L^1$-type P\'olya--Szeg\"o inequality holds on $\m$, namely
\[
    \int_{\m} |\nabla u^\sharp| \, \mathrm{d}v_g
    \le 
    \int_{\m} |\nabla u| \, \mathrm{d}v_g,
\]
for every $u \in W^{1,1}_0(\m)$, where $u^\sharp$ denotes the centered symmetric decreasing rearrangement of $u$, defined in \eqref{centered symmetrization}, and $W_0^{1,1}(\m)$ is the completion of the compactly supported smooth functions in the usual Sobolev norm. Then $\m$ satisfies the centered isoperimetric inequality \eqref{Cenetered-2}.
\end{theorem}

\medskip

A fundamental feature of the centered isoperimetric inequality is that it constitutes a strictly \emph{stronger} geometric principle than the Cartan--Hadamard conjecture. 
Indeed, if geodesic balls centered at the pole minimize the perimeter among all measurable sets of prescribed volume in the manifold, then a comparison with the Euclidean model yields the Euclidean isoperimetric inequality. In this sense, the centered isoperimetric inequality \textit{implies} the Cartan--Hadamard conjecture.

Following \cite{Yau}, we introduce the isoperimetric quotient associated with geodesic balls,
\begin{equation}\label{eq:isoperimetric-quotient}
    Q(r)
    = \frac{\text{Per}_{\m}(B_r(0))^n}{\text{Vol}_{\m}(B_r(0))^{n-1}},
\end{equation}
where $B_r(0)$ denotes the geodesic ball of radius $r$ centered at the pole $0$.
When the underlying warped product manifold has negative sectional curvature, the function $Q(r)$ is monotone increasing in $r$. 
Taking into account the infinitesimal Euclidean structure of the manifold (namely, the asymptotic expansions of volume and perimeter as $r \to 0$), we obtain our next result. 

First, we recall the notion of a \emph{warped product manifold}.

\begin{definition}[Warped product manifold]
    An $n$-dimensional Riemannian manifold $\m$ is called a warped product manifold, with the warping function $\psi$, if its metric $g$ can be written in the form
    \begin{equation*}
        g=\mathrm{d}r\otimes \mathrm{d}r+\psi(r)^2 g_{\sn},
    \end{equation*}
    where $r$ denotes the geodesic distance from a fixed pole $0\in \m$, $g_{\sn}$ is the standard metric on the unit sphere $\sn$ and $\psi: [0,\infty) \to [0,\infty)$ is a smooth function.
\end{definition}
For further details on this class of manifolds, we refer the reader to Section~\ref{prem}. We now proceed to state the next result.

\begin{theorem}\label{second}
Let $(\m,g)$ be an $n$-dimensional Cartan-Hadamard warped product manifold. 
Assume that $\m$ satisfies the centered isoperimetric inequality \eqref{Cenetered-2}.
Then $\m$ satisfies the Cartan--Hadamard conjecture; that is, the Euclidean isoperimetric inequality holds on $\m$.
\end{theorem}

\medskip

\begin{remark}
{\rm 
An inspection of the proof shows that, as an immediate consequence, the Cartan–Hadamard conjecture holds within the class of Cartan–Hadamard warped product manifolds when restricted to geodesic balls, regardless of their volume and of the dimension of the manifold. In other words, geodesic balls satisfy the Euclidean isoperimetric inequality without any constraint on their size. For a related result in the setting of compact manifolds, we refer to \cite[Theorem~3.4]{MJ}.
 }
\end{remark}

\medskip 

Despite its apparently simple formulation, the centered isoperimetric problem remains highly nontrivial even in low-dimensional settings. A complete resolution required substantial effort for elementary surfaces of revolution. In the seminal work \cite{CB}, Itai Benjamini and Jianguo Cao proved in 1996 that, on a paraboloid of revolution, the perimeter-minimizing region enclosing a prescribed area is necessarily a circle of revolution. This result was subsequently recovered by different techniques by Pierre Pansu \cite{Pansu}, Peter Topping \cite{Topping}, Frank Morgan, Michael Hutchings, and Hugh Howards \cite{FMH}, and Manuel Ritor\'e \cite{M Ritore}, each employing distinct variational or geometric arguments.

These works led to a systematic classification of isoperimetric regions for several new classes of rotationally symmetric surfaces. In particular, Benjamini and Cao established the centered isoperimetric property for complete planes of revolution with Gauss curvature non-increasing from the origin and convex at infinity. The convexity assumption was subsequently removed by Morgan, Hutchings, and Howards in \cite{FMH}, where the authors characterized isoperimetric regions in real projective planes of revolution under the sole assumption that the curvature is non-increasing. Further advances were made by Ritor\'e, who solved the isoperimetric problem for spheres of revolution exhibiting equatorial symmetry and whose Gauss curvature is either nonincreasing or nondecreasing from the equator toward the poles.

A striking feature emerging from these contributions is the recurring role played by \textit{curvature monotonicity}: in particular, the decreasing behavior of the curvature appears as a fundamental structural requirement. This phenomenon persists in higher dimensions and, as shown in the present work, the monotonicity of the curvature arises as a necessary condition for the validity of the centered isoperimetric inequality on Cartan–Hadamard warped product manifolds. Our approach relies on the stability of centered geodesic spheres, obtained through the second variation of the perimeter functional, tested against eigenfunctions corresponding to the first nonzero eigenvalue of the Laplace--Beltrami operator on $\sn$. The resulting analysis leads to the following theorem.

\begin{theorem}\label{third}
Let $(\m,g)$ be an $n$-dimensional Cartan--Hadamard warped product manifold. 
Assume that $\m$ satisfies the centered isoperimetric inequality \eqref{Cenetered-2}. 
Then the following geometric properties hold:

\begin{itemize}
    \item[(i)] The radial sectional curvature cannot be monotonically increasing as a function of the distance from the pole.
    
    \item[(ii)] At every point, the radial sectional curvature is less than or equal to the tangential sectional curvature.
    
    \item[(iii)] The tangential sectional curvature is monotone decreasing as a function of the distance
from the pole.
\end{itemize}
\end{theorem}

\medskip

\begin{remark}
{\rm 
   The second necessary condition coincides with the structural assumption imposed in \cite[Theorem~1.4]{Brendle}. In that work, it is further proved that, if this condition holds with strict inequality and is complemented by suitable mild regularity assumptions, then the geodesic spheres $\partial B_r(o)$ are the unique closed hypersurfaces with constant mean curvature.}
\end{remark}

Isoperimetric inequalities provide a fundamental link between the geometry of a manifold and analytic properties of differential operators defined on it. In particular, the minimization of perimeter under a volume constraint naturally leads to the notion of the Cheeger's constant, which quantifies the optimal isoperimetric ratio of subsets of the manifold. In his seminal work \cite{Cheeger}, Jeff Cheeger established a deep connection between this geometric quantity and spectral theory, proving that the first nonzero eigenvalue of the Laplace–Beltrami operator admits a lower bound in terms of the Cheeger's constant (see \cite{Chavel}). This result, now known as the Cheeger's inequality, reveals how isoperimetric properties control the spectral gap and has since become a cornerstone in geometric analysis. In \cite{Buser}, an upper bound for the smallest positive eigenvalue of the Laplacian on $\m$ in terms of the Cheeger constant is given. For some applications of the Cheeger's inequality we refer to \cite{Brooks 1, Brooks 2, Dodziuk 1, Dodziuk 2, Donnely, Osser}.

In this article we introduce a Cheeger-type isoperimetric quotient defined by restricting the class of admissible sets to geodesic balls centered at the pole:
\begin{equation}\label{Cheeger type quotient}
    I(r)=\frac{\text{Per}_{\m}(B_r(0))}{\text{Vol}_{\m}(B_r(0))}.
\end{equation}
Our goal is to derive a Cheeger-type lower bound expressed in terms of the asymptotic behavior of the isoperimetric quotient $I(r)$ as $r\to \infty$. To this end, we make essential use of Persson’s theorem (\cite{Persson}, see Section~\ref{prem}). The following result yields an estimate that, for a broad class of warped product manifolds, strictly improves upon the classical Cheeger inequality. Instances in which this yields a strict improvement over the classical Cheeger's inequality are discussed in Remark~\ref{strict improvement}.

\begin{theorem}\label{fourth}
Let $(\m,g)$ be a Riemannian manifold with warping function $\psi$, and let $I(r)$ denote the quantity introduced in \eqref{Cheeger type quotient}. 
Assume that the following conditions are satisfied:
\begin{itemize}
    \item[(i)] The limit 
    \[
        L := \lim_{r \to \infty} I(r)
    \]
    exists and is finite.
    
    \item[(ii)] 
    \[
        \lim_{r \to \infty} \left( \frac{\psi'(r)}{\psi(r)} \right)' = 0.
    \]
\end{itemize}
Then the bottom of the spectrum of the Laplace--Beltrami operator on $\m$ is given by
\begin{equation}\label{eq:lambda-explicit}
    \lambda_1(\m) = \frac{L^2}{4}.
\end{equation}
\end{theorem}

\medskip

\begin{remark}
{\rm 
Suppose that 
\begin{equation*}
    \lim_{r \to \infty}\left(\frac{\psi'(r)}{\psi(r)}\right)'=\tilde{L}<\infty
\end{equation*}
and that all the remaining assumptions of Theorem~\ref{fourth} are satisfied. Then the bottom of the spectrum of the Laplace–Beltrami operator on $\m$ is given by 
\begin{equation*}
    \lambda_1(\m)= \frac{L^2}{4}+\frac{(n-1)\tilde{L}}{2}.
\end{equation*}
The proof follows along the same lines as that of Theorem~\ref{fourth} and is therefore omitted.
}
\end{remark}

In our final result,  we establish an explicit lower bound for the first nonzero eigenvalue of the geodesic ball $B_r(0)$ on a broad class of manifolds. The argument begins with the derivation of an upper bound for the volume of an arbitrary measurable set $A\subset B_r(0)$ expressed in terms of an integral involving the geometric quantity $\frac{\varrho(x)\psi'(\varrho(x))}{\psi(\varrho(x))}$. This step relies on the assumption that the function $t\to \frac{t\psi'(t)}{\psi(t)}$ is monotone increasing; as shown in Remark~\ref{final remark}, this condition is satisfied by a wide class of manifolds. Finally, by selecting a suitable radial vector field and applying the divergence theorem in conjunction with the obtained volume estimate, we arrive at the desired result.

\begin{theorem}\label{fifth}
    Let $\m$ be an n-dimensional manifold with warping function $\psi$ such that $\frac{t\psi'(t)}{\psi(t)}$ is increasing for $t>0$. Let $B_r(0)$ be the geodesic ball of radius $r$ and centered at the pole. If $\lambda_1(B_r)$ denotes the first non-zero eigenvalue of $B_r(0)$ then
    \begin{equation*} 
        \lambda_1(B_r) \geq \frac{n^2}{4r^2}.
    \end{equation*}
\end{theorem}

\medskip

The paper is organized as follows. In Section~\ref{prem} we collect the necessary preliminaries and background material that will be used throughout the paper. Section~\ref{auxiliary} is devoted to the derivation of several results describing the interplay between different curvature quantities; these results are instrumental for the proofs of the main theorems and are also of independent interest. Finally, in Section~\ref{main theorem} we present the proofs of the main results, together with a number of further consequences and related observations.

\medskip

\section{Functional and Geometric analytic preliminaries} \label{prem}
In this section, we will discuss the warped product spaces and the geometric preliminaries, which will be required throughout the paper.
\medskip

\subsection{Riemannian warped product manifolds}

Let $\m$ be an $n$-dimensional Riemannian manifold. The manifold 
$\m$ is called a \emph{warped product manifold} if its metric can be expressed in the form
\begin{equation}\label{metric}
g = \mathrm{d}r \otimes \mathrm{d}r + \psi(r)^2 g_{\mathbb{S}^{n-1}},
\end{equation}
where $r$ denotes the geodesic distance from a fixed point 
$0\in \m$, called the \emph{pole} of the manifold. Here 
$\mathrm{dr}$ represents the radial direction, $g_{\sn}$ denotes the standard metric on the unit sphere $\sn$ and $\psi: [0,\infty) \to [0,\infty)$ is a smooth function, referred to as the \emph{warping function}. The geometry of the manifold is completely determined by the choice of $\psi$.

In this framework, every point $x \in \m\setminus\{0\}$ admits a polar coordinate $(r,\theta)\in (0,\infty)\times \sn$ where $r$ denotes the distance from the pole $0$ and $\theta$ represents the direction of the minimizing geodesic joining $0$ to $x$. For a detailed discussion of such manifolds, we refer to \cite[Section 3.10]{grig}. Manifolds of this type arise as an important subclass of warped product spaces; see, for example, \cite[Section 1.8]{AMR}.

\medskip 

To guarantee that the Riemannian metric determined by the radial function $\psi$ is smooth at the pole, suitable regularity conditions must be imposed on $\psi$. These conditions are both necessary and sufficient for the metric to extend smoothly across the pole. More precisely, we assume that
\begin{equation}\label{psi}
\psi \in C^\infty([0,\infty)), \qquad
\psi(r) > 0 \ \text{for } r>0, \qquad
\psi'(0)=1, \qquad
\psi^{(2k)}(0)=0 \quad \text{for all } k \in \mathbb{N}\cup{0}.
\end{equation}
While in many analytical considerations the conditions $\psi(0)=0$ and $\psi'(0)=1$ are sufficient to ensure that the metric exhibits the correct first-order behavior near the pole, the full set of conditions in \eqref{psi} is required in order to ensure smoothness of the metric and to avoid any loss of regularity at $r=0$.

If the warping function $\psi$ is defined on the entire interval $[0,\infty)$, then the corresponding manifold is complete and noncompact. In particular, the radial coordinate $r$ is defined globally and measures the geodesic distance from the pole. Moreover, if the manifold has nonpositive sectional curvature, then the \emph{Cartan–Hadamard theorem} implies that it is diffeomorphic to the Euclidean space $\rn$. More precisely, for every point $p\in \m$, the exponential map $$\mathrm{exp}_p: T_p\m\to \m$$
is a global diffeomorphism where $T_p\m$ is the tangent space at the point $p$.

Important examples of noncompact Riemannian warped product manifolds include the Euclidean space $\rn$ and the hyperbolic space $\hn$ with the warping functions $\psi(r)=r$ and $\psi(r)=\sinh r$ respectively.

In terms of the local coordinate system $\{x^i\}_{i=1}^{N},$ one can write
\begin{align*}
g=\sum g_{ij}{\rm d}x^i\dx^j.
\end{align*}
The Laplace-Beltrami operator $\Delta_g$ concerning the metric $g$ is defined as follows
\begin{align*}
\Delta_g:=\sum \frac{1}{\sqrt{\text{det }(g_{ij})}}\frac{\partial}{\partial x^i}\bigg(\sqrt{\text{det }g_{ij}}\:g^{ij}\frac{\partial}{\partial x^j}\bigg),
\end{align*}
where $(g^{ij})=(g_{ij})^{-1}$. Also, we denote $\nabla_g$ as the Riemannian gradient, and for functions $u$ and $v$, we have\begin{align*}
\langle \nabla_g u , \nabla_g v \rangle_g= \sum g^{ij}\frac{\partial u}{\partial x^i}\frac{\partial v}{\partial x^j}.
\end{align*}
For simplicity, we shall use the notation $|\nabla_g f| = \sqrt{g(\nabla_g f,\nabla_g f)}$. Due to our geometry function $\psi$, we can write these operators more explicitly. Namely, for $x\in\m$, we can write $x=(r, \theta)=(r, \theta_{1},\ldots, \theta_{N-1})\in(0,\infty)\times\sn$ and the Riemannian Laplacian of a scalar function $f$ on $\m$ is given by 
\begin{equation*}
	\lapg u (r, \theta)  =
	\frac{1}{(\psi(r))^2} \frac{\partial}{\partial r} \left[ ((\psi (r)))^{n-1} \frac{\partial u}{\partial r}(r, \theta) \right] \\
	+ \frac{1}{(\psi(r))^2 } \Delta_{\mathbb{S}^{n-1}} u(r, \theta),
\end{equation*}
where $\Delta_{\mathbb{S}^{n-1}}$ is the Riemannian Laplacian on the unit sphere $\mathbb{S}^{n-1}$. Also, let us recall that the Gradient in terms of the polar coordinate decomposition is given by
\begin{equation*}
	\gradg u(r, \theta)=\bigg(\frac{\partial u}{\partial r}(r, \theta), \frac{1}{\psi(r)}\nabla_{\sn}u(r, \theta)\bigg),
\end{equation*}
where $\nabla_{\sn}$ denotes the Gradient on the unit sphere $\sn$.

For any point $x\in\m$ and any $r>0$, we denote by $B_r(x):=\{y\in \m\,:\, \text{dist}(y,x)<r \}$ the geodesic ball in $\m$ with center at a fixed point $x$ and radius $r$. The Riemannian volume measure determined by $g$ in the coordinate frame $x\equiv (r,\theta)$ is given by the product measure
\begin{equation*}
    dV(x)=\psi(r)^{n-1}\dr\dsn,
\end{equation*}
where $\dsn$ denotes the $(n-1)$ dimensional measure on the unit sphere.

For any function $u\in L^1(\m)$, the polar coordinate decomposition can be written as follows
\begin{align*}
\int_{\m}u(x)\dvg=\int_{\sn}\int_{0}^{\infty}u(r,\theta)\:(\psi(r))^{n-1}\dr\dsn.
\end{align*}
In particular, the volume of geodesic balls centered at the pole reads as
\begin{equation}\label{ball volume}
    \text{Vol}_{\m}(B_r(0))=\uv\int_{0}^r\psi(t)^{n-1}\dt,
\end{equation}
where $\uv$ is the $(n-1)$ dimensional Hausdorff (surface) measure of the unit sphere $\sn$.


\subsection{Various types of curvatures on \texorpdfstring{$\m$}{M}}
It is known that there exists an orthonormal frame $\{F_{j}\}_{j = 1, \ldots, N}$ on $\m$, where $F_{N}$ corresponds to the radial coordinate and $F_{1}, \ldots, F_{N-1}$ correspond to the spherical coordinates, such that $F_{i} \wedge F_{j}$ diagonalize the curvature operator $\mathcal{R}$:  
\[
	\mathcal{R}(F_{i} \wedge F_{N}) = - \frac{\psi^{\prime \prime}}{\psi} \, F_{i} \wedge F_{N}, \quad i < N,
\]
\[
	\mathcal{R}(F_{i} \wedge F_{j}) = - \frac{(\psi^{\prime})^2 - 1}{\psi^2} \, F_{i} \wedge F_{j}, \quad i,j < N.
\]
The quantities
\begin{equation}\label{radial and tangential curvature}
		K_{\text{rad}}(r) := - \frac{\psi^{\prime \prime}(r)}{\psi(r)} \quad \text{and} \quad K_{\text{tan}}(r) := - \frac{(\psi^{\prime}(r))^2 - 1}{\psi(r)^{2}}
\end{equation}
coincide, respectively, with the sectional curvatures of planes containing the radial direction and of planes orthogonal to it.

The radial Ricci curvature $\mathrm{Ric}_{\text{rad}}:= \mathrm{Ric}(\dr, \dr)$ is the sum of all of the
$(n-1)$ sectional curvatures associated with planes containing $e\in T_x\m$ and it is given by using \eqref{radial and tangential curvature},
\begin{equation}\label{radial Ricci}
    \mathrm{Ric}_{\text{rad}}(r)=-(n-1)\frac{\psi''(r)}{\psi(r)}.
\end{equation}
The tangential counterpart is defined in any of the $(n-1)$ vectors orthogonal to $dr$ reads,
\begin{equation}\label{tangential Ricci}
    \mathrm{Ric}_{\text{tan}}(r):=-\frac{\psi''(r)}{\psi(r)}+(n-2)\frac{1-[\psi'(r)]^2}{\psi(r)^2}.
\end{equation}
The scalar curvature $S(x)$ at $x$ is the trace of the Ricci curvature tensor or, equivalently, the sum of
$\text{Ric}(e_i, e_i)$ over an orthonormal basis $\{e_i\}_{i=1...n}$ of $T_x(\m)$. In particular, for warped product manifolds, thanks to \eqref{radial Ricci} and \eqref{tangential Ricci} we obtain,
\begin{equation}\label{scalar curvature}
    S(x)\equiv S(r)=2(n-1)K_{\mathrm{rad}}(r)+(n-1)(n-2)K_{\mathrm{tan}}(r)
\end{equation}

There is a well-known relation between the scalar curvature and the asymptotic volume and perimeter of “small” geodesic balls. Indeed, for every $x\in \m$, we have the following expansions (see e.g. [\cite{Chavel}, Chapter
XII.8]):
\begin{align*}
    &\text{Vol}_{\m}(B_{\epsilon}(x))= \uv \epsilon^n\left(1-\frac{S(x)}{6(n+2)}\epsilon^2+O(\epsilon^3)\right),\\
    &\text{Per}_{\m}(B_{\epsilon}(x))=n\uv\epsilon^{n-1}\left(1-\frac{S(x)}{6n}\epsilon^2+O(\epsilon^3)\right).
\end{align*}

The mean curvature of a geodesic sphere of radius $r$ is given by
\begin{equation}\label{mean curvature}
    H(r)=(n-1)\frac{\psi'(r)}{\psi(r)}.
\end{equation}

\subsection{Perimeter on \texorpdfstring{$\m$}{M}} We recall in this section some basic facts of sets of finite perimeter in a Riemannian manifold $\m$. 
\begin{definition}\label{perimeter}
    The perimeter of a measurable set $E\subset \m$ inside an open set $\Omega$ is defined as
    \begin{equation*}
        |\partial \Omega|(E)=\mathrm{Per}_{\m}(E,\Omega) =\sup\left\{\int_E \mathrm{div}X\dvg : X \in \Xi_0^1(\Omega),\, \|X\|\leq 1 \right\},
    \end{equation*}
    where $\Xi_0^1(\Omega)$ is the space of vector fields of class $C^1$ in $\m$ with compact support inside $\Omega$ and $\|X\|$ is the supremum norm of $X$.
\end{definition}
When $\Omega=\m$ we write simply $\text{Per}_{\m}(E)$. A set $E$ has finite perimeter in $\Omega$ if $\text{Per}_{\m}(E,\Omega)<\infty$. We refer the reader to \cite{Giusti} and \cite{Maggi} for complete information on sets of finite perimeter. A good introduction to sets of finite perimeter in Riemannian manifolds can be found in \cite[section 1]{MPPP} and \cite[chapter 1]{Ritore}.

When $E$ is bounded and has $C^1$ boundary, the perimeter of $E$ in $\Omega$ coincides with the Riemannian measure of $\partial E \,\cap\,\Omega $. This is obtained immediately from the divergence theorem. It is clear that a measurable $E\subset \m$ is of finite perimeter if and only if its characteristic function $\chi_E$ is of bounded variation, which is defined as follows
\begin{equation*}
    |\mathrm{D}f|(\Omega)=\sup\left\{\int_{\Omega} f\text{div}X\dvg : X \in \Xi_0^1(\Omega),\, \|X\|\leq 1 \right\},
\end{equation*}
for any open subset $\Omega\subset \m$ and $f\in L^1(\Omega)$.
The set of all bounded variation function is denoted by $BV(\m)$.
\medskip

 An alternative definition of the perimeter on a Riemannian manifold can be formulated in terms of the Hausdorff measure. We recall for any $s \in [0,\infty)$ and $\delta \in (0,\infty]$ the $s$-dimensional Hausdorff measure of $E\subset \m$ is defined by
\begin{equation}\label{Hausdorff measure}
    \mathcal{H}^s(E):=\lim_{\delta\to 0}\mathcal{H}_{\delta}^s(E)=\sup_{\delta>0}\mathcal{H}_{\delta}^s(E),
\end{equation}
where,
\begin{equation*}
    \mathcal{H}_{\delta}^s(E):= \inf\left\{c(s)\sum_{i=1}^{\infty} (\text{diam}\,C_i)^s: E\subset \cup_{i=1}^{\infty} C_i, \,\text{diam}(C_i) \leq \delta\right\}.
\end{equation*}
Here $c(s)$ is an arbitrary positive constant only depending on $s$ and and $\text{diam}(C_i)=\sup\{\text{dist}(x,y): x,y \in C_i\}$.\par
Let us define the \textit{reduced boundary} $\partial^\star E$ of a $E\subset \m$ of locally finite perimeter. Let $\nu$ be the measurable unit normal of $E$ defined by 
\begin{equation*}
    \int_E \text{div}X\dvg=\int_{\m}\langle X,\nu\rangle\ \mathrm{d}\mu,
\end{equation*}
for every vector field $X$ with compact support of class $C^1$. 
We say that $x\in \partial^\star E$ if
\begin{itemize}
    \item[(i)] $\text{Per}_{\m}(E,B_r(x))>0$ for every $r>0.$
    \medskip 
    \item[(ii)] $\lim_{r \to 0}\nu_r(x)$ exists and equal to $\nu(x)$ where $\nu_r(x)$ is given by 
    \begin{equation*}
        \nu_r(x)=\frac{\int_{B_r(x)}\nu \ \mathrm{d}P}{\text{Per}_{\m}(B_r(x))},
    \end{equation*}
    and $\mathrm{d}P$ is the measure induced by the perimeter functional.
    \medskip
    \item[(iii)] $|\nu(x)|=1$.
\end{itemize}
\begin{theorem}[\cite{Ritore}, Theorem~1.39]
    Let $E\subset \m$ be a set of finite perimeter. Then
    \begin{equation*}
        \partial^\star E= \cup_{i=1}^\infty C_i \cup Z,
    \end{equation*}
    where $\mathrm{Per}(E,Z)=0$ and each $C_i$ is compact and is contained in the level set of a $C^1$ function with non-vanishing gradient. Moreover for every open $\Omega \subset \m$
    \begin{equation*}
        \mathrm{Per}(E,\Omega)= \mathcal{H}^{n-1}(\partial^\star E\, \cap \Omega),
    \end{equation*}
    where the Hausdorff measure is defined in \eqref{Hausdorff measure}.
\end{theorem}
We say that a sequence of measurable sets $\{E_i\}_{i\in \mathbb{N}}$ converges in $L^1(\m)$ or \textit{in measure} to a measurable set $E$  when the characteristic 
functions $\chi_{E_i}$ converge in $L^1(\m)$ to $\chi_E$. Based on this definition, several properties of sets of finite perimeter are listed below.
\begin{proposition}[Lower semicontinuity of perimeter]\label{LS}
    Let $\Omega\subset \m$ be an open set. Let $\{E_i\}_{i\in \mathbb{N}}$ be a sequence of sets of finite perimeter in $\Omega$ converging in $L^1_{\text{loc}}(\Omega)$  to a measurable set $E$. Then
    \begin{equation*}
        \mathrm{Per}(E,\Omega)\leq \liminf_{i\to \infty}\, \mathrm{Per}\,(E_i,\Omega).
    \end{equation*}
\end{proposition}
\begin{theorem}[Compactness]
   Let $\Omega\subset \m$ be a bounded open set with Lipschitz boundary. Let $\{E_i\}_{i\in \mathbb{N}}$ be a sequence of sets with uniformly bounded perimeters $\text{Per}(E_i, \Omega)$. Then we can extract a subsequence converging in $L^1(\Omega)$ to a set of finite perimeter $E\subset \Omega$.
\end{theorem}
The proofs of the above results can be found in the first chapter of \cite{Giusti}. Another property is the following
\begin{proposition}[\cite{MPPP}, Proposition 1.4]\label{approximation}
    For every $u\in BV(\m)$, there exists a sequence of $\{u_i\}_{i\in \mathbb{N}}\subset C_c^{\infty}(\m)$ such thst $u_i\to u$ in $L^1(\m)$ and 
    \begin{equation*}
        |\mathrm{D}u|(\m)= \lim_{ i\to \infty}\int_{\m} |\nabla u_i|\dvg.
    \end{equation*}
\end{proposition}
The coarea formula for sets of finite perimeter reads as follows
\begin{proposition}[\cite{Miranda}]
    Let $\Omega \subset \m$ be an open set and $u\in L^1_{\mathrm{loc}}(\Omega)$. Lettting $E_t=\{u>t\}$ we have
    \begin{equation*}
        \int_{\mathbb{R}}|\partial E_t|(\Omega)\dt= |\mathrm{D}u|(\Omega).
    \end{equation*}
    In case $u\in BV(\m)$, then $E_t$ has finite perimeter in $\Omega$ for a.e. $t\in \mathbb{R}$.
\end{proposition}

\subsection{Isoperimetric inequalities in manifold} The isoperimetric inequality is a fundamental geometric principle relating the perimeter of a set to its enclosed volume, with deep connections to curvature, spectral theory, and analysis on manifolds.
\begin{definition}
    The isoperimetric profile of $\m$ is the function $I_{\m}$ that assigns, to each $v\in (0,|\m|)$, the value 
    \begin{equation*}
        I_{\m}(v)= \inf\left\{\mathrm{Per}_{\m}(E): |E|=v\right\}.
    \end{equation*}
    A set $E\subset \m$ is called isoperimetric region if
    $$\mathrm{Per}_{\m}(E)=I_{\m}(|E|).$$
\end{definition}
The classical isoperimetric inequality in the Euclidean space states that round balls are the unique isoperimetric regions in $\mathbb{R}^N$. Regularity results for sets minimizing perimeter under a volume constraint were established by Morgan. In particular, in Corollaries~3.7 and~3.8 of \cite{Morgan}, he proved the following:

\begin{theorem}
    Let $E$ be a measurable set of finite volume minimizing perimeter under a volume constraint in a smooth $m$-dimensional Riemannian manifold $\m$. Then
    \begin{itemize}
        \item[(i)] If $n\leq 7$ then the boundary $S$ of $E$ is a smooth hypersurface.
        \medskip
        \item[(ii)] If $n>7$ then the boundary of $E$ is the union of a smooth hypersurface $S$ and a closed singular set $S_0$ of Hausdorff dimension at most $n-8$.
    \end{itemize}
\end{theorem}
In this article we will be focusing on two types of isoperimetric inequalities: the \emph{centered isoperimetric inequality} and the \emph{Cartan-Hadamard conjecture}.

\medskip 
\textbf{Centered isoperimetric inequality (CII)}: A Riemannian manifold is said to satisfy the CII if the following inequality holds
\begin{equation}\label{Cenetered}
    \text{Per}_{\m}(B_r(0))\leq \text{Per}_{\m}(\Omega),
\end{equation}
where $\Omega$ is any bounded, Borel measurable set in $\m$ and $B_r(0)$ is the geodesic ball centered at the pole of radius $r>0$ and have the same volume as $\Omega$.
If 
\begin{equation*}
    G(r)=\int_0^r\psi(t)^{n-1}\dt,
\end{equation*}
then the centered isoperimetric inequality is equivalent to the following inequality
\begin{equation}\label{CII}
    \text{Per}_{\m}(\Omega) \geq \uv\left[\psi \left(G^{-1}\left(\frac{\text{Vol}_{\m}(\Omega)}{\uv}\right)\right)\right]^{n-1}.
\end{equation}

\medskip

Under the assumption that the warping function $\psi$ is a \textit{convex} function, the manifold $\m$ becomes a Cartan-Hadamard manifold, i.e., a complete, simply connected Riemannian manifold with nonpositive sectional curvature. For further details, we refer the reader to 
\cite{GW}. In particular, $\psi$ satisfies

\begin{align*}
    \psi^{\prime \prime}(r)\geq 0 \quad \text{in } (0,+\infty),\,
\end{align*}
which in turn implies 
\begin{align*}
    \psi'(r)\geq 1 \quad \text{in } (0,+\infty).\,
\end{align*}

\medskip

\textbf{Cartan-Hadamard conjecture} [Aubin \cite{Aubin}]: Let $\m$ be an $n$-dimensional Cartan-Hadamard manifold. The conjecture says that the Euclidean isoperimetric inequality holds on $\m$: for every bounded, measurable set $\Omega\subset \m$ it holds that
\begin{equation}\label{CH-Cojecture}
    \text{Per}_{\m}(\Omega)\geq n|B(0,1)|^{\frac{1}{n}}[\text{Vol}_{\m}(\Omega)]^{\frac{n-1}{n}}
\end{equation}
where $|B(0,1)|$ is the Euclidean volume of the unit ball. Furthermore equality holds if and only if  $\Omega$ is isometric to a ball in $\rn$ (up to a set of volume zero).

\medskip

It turns out, via approximation theory, that the Cartan-Hadamard conjecture is equivalent to the Sobolev inequality for $p=1$:
\begin{equation}\label{L^1 sobolev}
    \|f\|_{L^{\frac{n}{n-1}}}\leq \frac{1}{n|B(0,1)|^{\frac{1}{n}}}\|\nabla f\|_{L^1(\m)},\quad \forall f\in W^{1,1}_0(\m)
\end{equation}
It is known that \eqref{CH-Cojecture} holds with some constant $C_n$ (and hence \eqref{L^1 sobolev} with $\frac{1}{C_n}$) [see \cite{Hebey}, Lemma 8.1 and Theorem 8.3 ], and using the infinitesimally Euclidean structure of any (smooth) Riemannian manifold, we can say that the value of $C_n$ is at most $n|B(0,1)|^{\frac{1}{n}}$. Whether the maximality of $C_n$ is achieved is the main content of the conjecture. We refer to \cite[Theorem 1.1]{Muratori-Soave} for some rigidity results pertaining to the Sobolev inequality in Cartan-Hadamard manifolds under the assumption of the Cartan-Hadamard conjecture.

\subsection{Symmetrization in manifold}
We consider all the measurable functions $f:\m\to\mathbb{R}$ whose superlevel sets are finite, i.e.,
\begin{equation*}
    \text{Vol}_{\m}\{x\in \m: |f(x)|>t\}<\infty,\quad \forall\, t>0.
\end{equation*}
Such functions are known as \textit{admissible functions}. For any such $f$, the distribution function $\mu_f$ is defined by
\begin{equation*}
    \mu_f(t)=\text{Vol}_{\m}\{x\in \m: |f(x)|>t\}.
\end{equation*}
It is clear that $\mu_f$ is non-increasing and right continuous. 

The \textit{Hardy-Littlewood rearrangement} of $f$ is defined as the generalized inverse of the distribution function $\mu_f$:
\begin{equation*}
    f^\star(s)=\sup\{t\geq0: \mu_f(t) > s\},
\end{equation*}
or equivalently,
\begin{equation*}
    f^\star(s)=\int_0^{\infty} \chi_{\{\mu_f>s\}}(t)\dt.
\end{equation*}
The Hardy-Littlewood rearrangement of a function is non-increasing, right continuous, lower semicontinuous, and has the same distribution function (equimeasurable) with the original function. For the details, see \cite{bnstn}. 

The \textit{Schwarz symmetrization} of $f$ is defined by
\begin{equation}\label{centered symmetrization}
    f^\sharp(x)= f^\star(\text{Vol}_{\m}(B_r(0))).
\end{equation}
$f^\sharp$ is equimeasurable with $f$, radially non-increasing and for each $x\in \m$, the map $t \to f^\sharp (tx)$ is right continuous on $(0,\infty)$. These conditions are characterizing property for the function $f^\sharp$ as well (see \cite{bnstn}, chapter 1).

From the Cavalieri principle and equimeasurability, it follows that, for any Borel measurable function $F: [0,\infty)\to[0,\infty)$
\begin{equation*}
\int_{\m}F(|f|)\dvg=\int_{\m}F(f^\sharp)\dvg=\int_0^{\infty}F(f^\star)\dt.
\end{equation*}
For any admissible $f,g$ we have the following Hardy-Littlewood inequality :
\begin{equation}\label{HL inequality}
    \int_{\m}fg\dvg\leq\int_{\m}f^\sharp g^\sharp\dvg=\int_0^{\infty}f^\star g^\star\dt
\end{equation}
Moreover, $f\to f^\sharp$ is $L^1$-nonexpansive in the sense that,
\begin{equation}\label{nonexpansive}
    \|f^\sharp-g^\sharp\|_{L^1(\m)}\leq \|f-g\|_{L^1(\m)}.
\end{equation}
We refer to \cite{BS} for the details regarding \eqref{HL inequality} and \eqref{nonexpansive}.

Now we recall the most important inequality in the rearrangement theory, the \textit{P\'olya-Szeg\"o inequality}. \vspace{-15pt}
\begin{definition}[Manifold-Manifold type]
    A noncompact manifold $\m$ with $n\geq 2$ is said to satisfy the Manifold-Manifold $L^2$-Pol\'ya-Szeg\"o inequality if for any admissible $u\in W^{1,2}_0(\m)$, $u^\sharp\in W^{1,2}_0(\m)$, and the inequality
    \begin{equation}\label{mmps}
        \int_{\m}|\gradg u^\sharp|^2\dvg\leq \int_{\m}|\gradg u|^2\dvg
    \end{equation}
    holds true.
\end{definition}
\begin{definition}[Manifold-Euclidean type, \cite{DHV}]
    A noncompact manifold $\m$ with $n\geq 2$ is said to satisfy the Manifold-Euclidean $L^2$-Pol\'ya-Szeg\"o inequality if for any admissible $u\in W^{1,2}_0(\m)$, $u^*\in W^{1,2}_0(\rn)$ and the inequality
    \begin{equation}\label{meps}
        \int_{\rn}|\nabla u^*|^2\dx\leq \int_{\m}|\gradg u|^2\dvg
        \end{equation}
        holds true where $u^*:\rn\to[0,\infty)$ is the \textit{Euclidean rearrangement function}  which is radially symmetric, non-increasing in $|x|$, and for every $t>0$ is defined by,
        \begin{equation*}
            \mathrm{Vol}_{\rn}\{x\in\rn: u^*(x)>t\}=\mathrm{Vol}_{\m}\{x\in \m: |u(x)|>t\}.
        \end{equation*}
\end{definition}
The co-area formula, coupled with the centered isoperimetric inequality and the Cartan-Hadamard conjecture, implies \eqref{mmps} and \eqref{meps} respectively. See \cite{MV} and \cite{AK} respectively. 

\subsection{Spectrum of the Laplace-Beltrami operator and spherical harmonics}
For all open subset $\Omega \subset \m$, we denote by $\lambda_1(\Omega)$ the first eigenvalue of the Laplace-Beltrami operator $\Delta_g$ in $\Omega$ with zero Dirichlet boundary condition on $\partial \Omega$, i.e.,
\begin{equation}\label{first eigenvalue Rayleigh quotient}
    \lambda_1(\Omega)=\inf_{u\in C_c^{\infty}(\Omega)\setminus \{0\}}\frac{\int_{\Omega}|\nabla u|^2\dvg}{\int_{\Omega}u^2\dvg}.
\end{equation}
When $\Omega$ is bounded and smooth enough, the existence of the positive eigenfunction corresponding to $\lambda_1(\Omega)$ of the Laplace-Beltrami operator can be assured. On compact manifolds, the spectrum consists of a purely discrete sequence
\begin{equation*}
    0=\lambda_0<\lambda_1\leq \lambda_2\leq \cdots\uparrow \infty,
\end{equation*}
while on noncompact manifolds, the presence of continuous spectrum is typical and strongly influenced by the geometry at infinity. We refer to \cite{CM} for the characterization of having a discrete spectrum of the Laplacian on a class of manifolds.

Two important results regarding the lower bound on the bottom of the spectrum of the laplacian are the following:
\begin{theorem}[Cheeger \cite{Cheeger}]
    Let $\m$ be a noncompact Riemannian manifold of dimension $n\geq 2$, possibly having nonempty boundary and possibly having nonempty closure. For any connected $\Omega\subset \m$ with compact closure and piecewise smooth boundary, there holds
    \begin{equation*}
        \lambda_1(\Omega) \geq \frac{h(\Omega)^2}{4},
    \end{equation*}
    where the Cheeger contant $h(\Omega)$ is defined by
    \begin{equation*}
        h(\Omega)=\inf_{\Omega '\subset \Omega}\frac{\text{Per}_{\Omega}(\Omega ')}{\text{Vol}_{\m}(\Omega ')}
    \end{equation*}
    with $\Omega '$ being any open submanifold of $\Omega$ with compact closure in $\Omega$ and smooth boundary.
\end{theorem}

As a consequence of the Cheger's inequality we derive the McKean's inequality [see \cite{Chavel}]
\begin{theorem}[McKean \cite{McKean}]
    Let $\m$ be a simply connected manifold such that all of the sectional curvatures are less than or equal to a fixed negative constant $\kappa$. Then for any connected $\Omega\subset \m$ with compact closure and piecewise smooth boundary there holds
    \begin{equation*}
        \lambda_1(\Omega)\geq -\frac{(n-1)^2 \kappa}{4}.
    \end{equation*}
\end{theorem}

The bottom of the essential spectrum of certain Schr\"odinger operators can be stated explicitly. This is the famous \textit{Persson's theorem} \cite{Persson}. A simple form of the result is stated below.
\begin{theorem}[Persson]
    Let $H$ be the Schr\"odinger operator defined by,
    \begin{equation*}
        H:=-\frac{d^2}{dr^2}+W(r)
    \end{equation*}
    on $L^2(0,\infty)$ with Dirichlet boundary condition at $r=0$. Suppose the potential $W(r)$ is locally integrable and bounded from below, and let the limit at infinity exist finitely:
    \begin{equation*}
        L:= \lim_{r \to \infty}W(r)<\infty.
    \end{equation*}
    Then the bottom of the essential spectrum $\sigma_{\text{ess}}(H)$ is given by 
    \begin{equation*}
        L=\inf\sigma_{\text{ess}}(H).
    \end{equation*}
\end{theorem}
\medskip
Finally, we recall some basic facts from spherical harmonics. Let $\{Y_{k,m}(\theta)\}$ be an complete orthonormal system of spherical harmonics on $\sn$ corresponding to the eigenvalue $\lambda_k$ with multiplicity $m$. This satisfies
\begin{equation*}
    -\Delta_{\sn} Y_{k,m}(\sigma)=\lambda_k Y_{k,m}(\sigma).
\end{equation*}
If we take an arbitrary test function $u \in C^{\infty}_{c}(M)\setminus \{0\}$ then by the spherical harmonic decomposition of $u(r,\cdot)$, we have, for each fixed $r$,
\begin{equation*}
    u(r,\theta)= \sum_{k,m} a_{k,m}(r)Y_{k,m}(\theta),
\end{equation*}
where,
\begin{equation*}
    a_{k,m}(r)=\int_{\sn} u(r,\theta)Y_{k,m}(\theta)\mathrm{d}\sigma(\theta)
\end{equation*}
and the convergence is in the $L^2(\sn, d\sigma)$. The Parseval identity gives, for each fixed $r$,
\begin{equation}\label{P}
    \int_{\sn} |u(r,\theta)|^2\ \mathrm{d}\sigma(\theta)= \sum_{k,m}|a_{k,m}(r)|^2.
\end{equation}
Using Tonelli and \eqref{P} we compute,
\begin{equation}\label{denominator}
    \int_{\m}|u|^2\dvg=\int_0^{\infty}\psi(r)^{n-1}\left(\int_{\sn}|u(r,\theta)|^2\ \mathrm{d}\sigma\right)\dr=\sum_{k,m}\int_0^{\infty}|a_{k,m}(r)|^2\psi(r)^{n-1}\dr.
\end{equation}
In polar coordinates,
\begin{equation}\label{polar gradient}
    \int_{\m}|\gradg u|^2\dvg=\int_0^{\infty}\int_{\sn}\left((\partial_ru)^2+\frac{1}{\psi(r)^2}|\nabla_{\sn}u|^2\right)\psi(r)^{n-1}\ \mathrm{d}\sigma\dr.
\end{equation}
It is easy to see from the Parseval identity,
\begin{equation}\label{first simplification}
    \int_0^{\infty}\int_{\sn}(\partial_ru)^2\psi(r)^{n-1}\ \mathrm{d}\sigma\dr=\sum_{k,m}\int_0^{\infty}|a'_{k,m}(r)|^2\psi(r)^{n-1}\dr.
\end{equation}
Orthonormality of eigenfunctions with the fact that $\int_{\sn}|\nabla_{\sn}Y_{k,m}|^2\ \mathrm{d}\sigma=\lambda_k$ leads to,
\begin{equation}\label{second simplification}
    \int_0^{\infty}\int_{\sn} \frac{1}{\psi(r)^2}|\nabla_{\sn}u|^2\psi(r)^{n-1}\ \mathrm{d}\sigma\dr=\sum_{k,m}\int_0^{\infty}\frac{\lambda_k}{\psi(r)^2}|a_{k,m}(r)|^2\psi(r)^{n-1}\dr.
\end{equation}
Putting \eqref{first simplification} and \eqref{second simplification} in \eqref{polar gradient} we get
\begin{equation}\label{gradient simplification with sh}
    \int_{\m}|\gradg u|^2\dvg=\sum_{k,m}\int_0^{\infty}\left(|a'_{k,m}(r)|^2+\frac{\lambda_k}{\psi(r)^2}|a_{k,m}(r)|^2\right)\psi(r)^{n-1}\dr.
\end{equation}
\medskip

\section{Some Auxiliary results}\label{auxiliary}
In this section, we will prove some intermediate results which are of independent interest as well as have applications in the proof of the main theorems. We start by proving some lemmas which give relationships between different types of curvatures of $\m$.

\begin{lemma}\label{k to S}
   Let $\m$ be a Cartan-Hadamard manifold with the warping function $\psi$. If $K_{\mathrm{rad}}(r)$ is increasing (or, deecreasing) then $K_{\text{tan}}(r)$ and hence $S(r)$ is increasing (or, decreasing) where the curvature terms are defined in \eqref{radial and tangential curvature} and \eqref{scalar curvature}.
\end{lemma}
\begin{proof}
    Observe that,
    \begin{equation*}
        K_{\mathrm{tan}}'(r)=2\frac{\psi'(r)}{\psi(r)}\left(K_{\mathrm{rad}}(r)-K_{\mathrm{tan}}(r)\right).
    \end{equation*}
Now,
\begin{align*}
    (\psi(r)^2(K_{\mathrm{rad}}(r)-K_{\mathrm{tan}}(r)))'
    &= \; (\psi(r)^2K_{\text{rad}}'(r))-(\psi(r)^2K_{\text{tan}}'(r))\\
    & = \; 2\psi(r)\psi'(r)K_{\text{rad}}(r)+\psi(r)^2K_{\text{rad}}'(r)-2\psi(r)\psi'(r)K_{\text{rad}}(r)\\
  &  =\psi(r)^2K_{\text{rad}}'(r).
\end{align*}
Integrating from $0$ to $r$ and noting that $\psi(0)=0$ we get
\begin{equation*}
    \psi(r)^2(K_{\text{rad}}(r)-K_{\text{tan}}(r))=\int_0^r\psi(s)^2K_{\text{rad}}'(s)\ds.
\end{equation*}
If we assume $K_{\text{rad}}(r)$ is increasing then $K_{\text{rad}}'(r)\geq 0$ hence,
\begin{equation*}
    \psi(r)^2(K_{\text{rad}}(r)-K_{\text{tan}}(r))=\int_0^r\psi(s)^2K_{\text{rad}}'(s)\ds \geq 0,
\end{equation*}
which implies
\begin{equation*}
    K_{\text{rad}}(r) \geq K_{\text{tan}}(r).
\end{equation*}
Since $\psi''(r) \geq 0$ we have $\psi'(r)$ increasing. Since $\psi'(0)=1$ and $\psi(r)>0$ for $r>0$ we have
\begin{equation*}
    \frac{\psi'(r)}{\psi(r)}>0,\,\, \forall r >0.
\end{equation*}
Hence,
\begin{equation*}
    K_{\text{tan}}'(r)=2\frac{\psi'(r)}{\psi(r)}(K_{\text{rad}}(r)-K_{\text{tan}}(r)) \geq 0,
\end{equation*}
which proves $K_{\text{tan}}(r)$ is increasing.\par Finally,
\begin{equation*}
    S'(r)= 2(n-1)K_{\text{rad}}'(r)+(n-1)(n-2)K_{\text{tan}}'(r) \geq 0,
\end{equation*}
which proves that $S(r)$ is increasing as well.\par If $K_{\text{rad}}(r)$ is decreasing then $K_{\text{rad}}'(r) \leq 0$ and everything goes exactly similarly with obvious necessary changes.
\end{proof}

The next result gives a simple criterion ensuring that the derivative of a positive function satisfying a suitable differential inequality can change sign at most once
\begin{proposition}\label{monotone}
    Let $f \in C^{\infty}(0,\infty)$ be strictly positive and assume there exists a constant $C>0$ such that 
    \begin{equation*}
        f''\geq -Cf(r)f'(r) \quad \forall r>0,
    \end{equation*}
    then $f'$ can change sign at most once in $(0,\infty)$.
\end{proposition}
\begin{proof}
    Let 
    \begin{equation*}
        \Phi (r)=\int_{r_0}^rCf(s)\ds,
    \end{equation*}

where $r_0$ is arbitrary but fixed. Since $f>0$, $\Phi$ is strictly increasing and smooth. Multiplying the given inequality by $e^{\Phi(r)}$ and recognizing the complete expression as an exact derivative, we get
\begin{equation*}
    \left(e^{\phi(r)}f'(r)\right)' \geq 0.
\end{equation*}
Hence $r \to e^{\Phi(r)}f'(r)$ is a nondecreasing function on $(0,\infty)$.
Suppose $f'(r_{\star})>0$ for some $r_{\star}$. Using the monotonicity of $e^{\Phi(r)}f'(r)$ we get 
\begin{equation*}
    e^{\Phi(r)}f'(r) \geq e^{\Phi(r_{\star})}f'(r_{\star}) >0\quad \forall\, r \geq r_{\star}.
\end{equation*}
Since $e^{\Phi(r)}>0$ for every $r>0$, we conclude
\begin{equation}\label{positivity}
    f'(r)>0\quad \forall\, r> r_{\star}.
\end{equation}
Similarly, by reversing the inequality direction in $r$, we can conclude that if $f'(r_{\star})<0$ for some $r_{\star}$ then $f'(r)<0$ for every $r<r_{\star}$.\par Suppose, for contradiction, that $f'$ changes sign at least twice. Then there exists 
\begin{equation*}
    0<r_1<r_2<r_3
\end{equation*}
such that 
\begin{equation*}
    f'(r_1)>0,\, f'(r_2)<0,\,f'(r_3)>0.
\end{equation*}
From $f'(r_1)>0$ and \eqref{positivity} we have
\begin{equation*}
    f'(r)>0\quad \forall\,r>r_1,
\end{equation*}
which contradicts $f'(r_2)<0$. The opposite pattern
\begin{equation*}
    f'(r_1)<0,\, f'(r_2)>0,\, f'(r_3)<0
\end{equation*}
is ruled out analogously.
\end{proof}

The preceding proposition establishes that the monotonicity of the mean curvature $H(r)$ is governed by the monotonicity of the radial sectional curvature $K_{\text{rad}}(r)$ in a warped product manifold.
\begin{lemma}\label{radial to mean}
    Let $\m$ be a warped product manifold, and let $K_{\text{rad}}(r)$ and $H(r)$ be the same as defined in \eqref{radial and tangential curvature} and \eqref{mean curvature} respectively. Then the following statements hold:
    \begin{itemize}
        \item[(i)] If $K_{\text{rad}}(r)$ is decreasing on $(0,\infty)$, then $H(r)$ can change its monotonicity at most once on $(0,\infty)$. 
        \medskip
        \item[(ii)] If $K_{\text{rad}}(r)$ is increasing on $(0,\infty)$ then $H(r)$ is decreasing on $(0,\infty)$.
    \end{itemize}
\end{lemma}
\begin{proof}
    By direct computation, we get,
    \begin{equation}\label{Riccati equation}
       H''(r)= -(n-1)K_{\text{rad}}'(r)-\frac{2}{n-1}H(r)H'(r)\quad \forall\, r>0 .
    \end{equation}
By assumption $K_{\text{rad}}'(r)\leq 0$. Hence, from \eqref{Riccati equation} we get
\begin{equation*}
    H''(r) \geq -\frac{2}{n-1}H(r)H'(r).
\end{equation*}
Now applying proposition~\ref{monotone}, we conclude $H'(r)$ can change sign at most once. Hence $H(r)$ has at most one monotonicity change.\par
For the second part, we compute
\begin{equation}\label{mean derivative}
    \left(\frac{\psi'(r)}{\psi(r)}\right)'=-\frac{(\psi'(r))^2+K_{\text{rad}}(r)(\psi(r))^2}{(\psi(r))^2}.
\end{equation}
Now,
\begin{equation*}
    ((\psi'(r))^2+K_{\text{rad}}(r)(\psi(r))^2)'=K_{\text{rad}}'(r)(\psi(r))^2.
\end{equation*}
Since $K'_{\text{rad}}(r)\geq 0$ we get the function $(\psi'(r))^2+K_{\text{rad}}(r)(\psi(r))^2$ to be monotone increasing. This implies the positivity of the numerator in \eqref{mean derivative} since it starts with the value $1$. The conclusion then follows immediately.
\end{proof}
\medskip
In section~\ref{prem}, we have defined the perimeter of a set in a manifold (cf. definition~\ref{perimeter}). The explicit form of the perimeter for a general type of domain is not known. In fact, it seems quite difficult to deduce 
a nice compact form, possibly as a function of the warping function $\psi$, if the geometry of the domain is too rough. Nevertheless, for a special type of domain, the perimeter can be given explicitly in terms of the warping function $\psi$.

We define a bounded, measurable set $\Omega$ to be \textit{radial graph} if there exists a $C^1$ function $u: \sn\to (0,\infty)$ such that 
\begin{equation*}
    \Omega= \{(\varrho, \theta): 0 \leq \varrho \leq u(\theta)\}
\end{equation*}
The boundary $\partial \Omega$ is the hypersurface defined by the graph of $u$:
\begin{equation*}
    \partial \Omega=\{(u(\theta),\theta): \theta \in \sn\}
\end{equation*}
\vspace{-18pt}
\begin{lemma}
    Let $\m$ be an n-dimensional manifold with warping function $\psi$. The perimeter of a bounded radial graph set with the associated function $u \in C^1(\sn)$ is given by 
    \begin{equation*}
        \mathrm{Per}_{\m}(\Omega)=\int_{\sn}(\psi(u))^{n-1}\sqrt{1+\frac{|\nabla_{\sn}u|^2}{(\psi(u))^2}}\ \mathrm{d}S.
    \end{equation*}
\end{lemma}

\begin{proof}
    Let $h$ denote the standard metric on $\sn$, with local components
    \begin{equation*}
        h=h_{ij}(\theta)d\theta^id\theta^j,\quad h^{ij}=(h_{ij})^{-1}.
    \end{equation*}
    Since $u$ is $C^1$, the hypersurface $\partial \Omega$ is $C^1$ and $\Omega$ is a set of finite perimeter. Hence
    \begin{equation*}
        \text{Per}_{M}(\Omega)=\mathcal{H}^{n-1}(\partial \Omega).
    \end{equation*}
    We define the parameterization 
    \begin{equation*}
        F: \sn\to \m, \quad F(\theta)=(u(\theta),\theta).
    \end{equation*}
    In local coordinates $(\theta_1,\cdots, \theta^{n-1})$ the tangent vectors to $\partial \Omega$ 
    \begin{equation*}
        X_i:=\partial_i F=\frac{\partial u}{\partial \theta_i}\partial_r+\partial_i,\quad i=1,\cdots, (n-1),
    \end{equation*}
    where $\partial_r$ is the radial coordinate vector and $\partial_i$  are the coordinate vector fields on $\sn$.
Using the structure of the metric $g$,
\begin{equation*}
    g(\partial_r, \partial_r)=1,\quad g(\partial_r, \partial_i)=0,\quad g(\partial_i,\partial_j)=\psi(r)^2h_{ij}.
\end{equation*}
We compute the induced metric $g^{\partial \Omega}$ on $\partial \Omega$: 
\begin{equation*}
    g_{ij}^{\partial \Omega}= g(X_i, X_j)= u_iu_j+\psi(u(\theta))^2h_{ij}(\theta).
\end{equation*}
Let $A=\psi(u(\theta))^2h_{ij}(\theta)$ and $v_{,i}=u_i$ then
\begin{equation*}
    g_{ij}^{\partial \Omega}=A+vv^{T}
\end{equation*}
Since $A$ is positive definite, we apply the matrix determinant identity
\begin{equation*}
    \det(A+vv^T)=\det(A)(1+u^TA^{-1}u).
\end{equation*}
   Now,
   \begin{equation*}
     \det(A)= \psi(u)^{2(n-1)} \det(h_{ij}),\quad A^{-1}=\psi(u)^{-2}h^{ij}.
   \end{equation*}
Therefore,
\begin{equation*}
    \det(g_{ij}^{\partial \Omega})=\psi(u)^{2(n-1)} \det(h_{ij})\left(1+\frac{h^{ij}u_{,i}u_{,j}}{\psi(u)^2}\right).
\end{equation*}
Taking square roots,
\begin{equation*}
    \sqrt{\det(g_{ij}^{\partial \Omega})}= \psi(u)^{n-1}\sqrt{\det(h_{ij})}\sqrt{1+\frac{|\nabla_{\sn}u|^2}{\psi(u)^2}},
\end{equation*}
where
\begin{equation*}
    |\nabla_{\sn}u|^2:= h^{ij}u_{,i}u_{,j}.
\end{equation*}
Since 
\begin{equation*}
    dS(\theta):=\sqrt{\det(h_{ij})}\ \mathrm{d}\theta
\end{equation*}
is the standard surface measure on $\sn$, the $n-1$-dimensional Hausdorff measure on $\partial \Omega$ is
\begin{equation*}
    \text{Per}_{\m}(\Omega)=\mathcal{H}^{n-1}(\partial \Omega)= \int_{\sn}(\psi(u(\theta)))^{n-1}\sqrt{1+\frac{|\nabla_{\sn}u(\theta)|^2}{(\psi(u(\theta)))^2}}\ \mathrm{d}S(\theta).
\end{equation*}
\end{proof}
\begin{remark}
\rm
    If $u(\theta)=R$ is constant then $\nabla_{\sn}u=0$ and 
    \begin{equation*}
        \text{Per}_{\m}(\Omega)=\psi(R)^{n-1}\uv
    \end{equation*}
    which is exactly the area of the geodesic sphere of radius $R$ in the manifold.\rm
\end{remark}

Now, we turn to the isoperimetric problem. If the centered geodesic ball $B_R(0)$ is an isoperimetric set, it must be a stable critical point of the perimeter functional under the volume-preserving deformation. This implies that the second variation of the perimeter must be nonnegative for any admissible variation.

Let $\Sigma_R=\partial B_R$ be the boundary sphere. We consider a smooth variation of the boundary defined by the normal vector field $X=u\nu$, where $\nu$ is the outward unit normal and $u: \Sigma_R \to \mathbb{R}$ is scalar function. To preserve volume upto first order, the variation $u$ must satisfy 
\begin{equation*}
    \int_{\Sigma_R}u\ \mathrm{d}A=0.
\end{equation*}
The second variation of the perimeter, denoted by $Q(u,u)$, is given by the standard formula in Riemannian geometry
\begin{equation}\label{second variation form}
    Q(u,u)= \int_{\Sigma_R} (|\nabla^{\Sigma}u|^2-(\mathrm{Ric}(\nu,\nu)+\|A\|^2)u^2)\ \mathrm{d}A,
\end{equation}
where,
\begin{itemize}
    \item $\nabla^{\Sigma}$ is the gradient on the sphere.
    \item $\text{Ric}(\nu,\nu)$ is Ricci curvature of the ambient manifold in the radial direction.
    \item $\|A\|^2$ is the squared norm of the second fundamental form of the sphere.
\end{itemize}
\textit{Stability condition}: For $B_R(0)$ to be the a minimizer, we must have $Q(u,u) \geq 0$ for all smooth $u$ with mean zero.\par The boundary $\Sigma_R$ is a level set of the distance function $r$. The shape operator (Weingarten map) $A$ for a warped product manifold is a scalar multiple of the identity:\begin{equation*}
    A(X)=\frac{\psi'(R)}{\psi(R)}X,\,\, \text{for any}\, X\in T\Sigma_R.
\end{equation*}
The squared norm is the sum of the squares of the principle curvatures. Since $\Sigma_R$ has dimension $n-1$
\begin{equation}\label{second fundamental form}
    \|A\|^2=\Sigma_{i=1}^{n-1}\left(\frac{\psi'(R)}{\psi(R)}\right)^2=(n-1)\left(\frac{\psi'(R)}{\psi(R)}\right)^2.
\end{equation}
The Ricci curvature in the radial direction $\nu=\partial_r$ for a warped product space is determined by the warping function
\begin{equation*}
    \text{Ric}(\partial_r,\partial_r)=-(n-1)\frac{\psi''(R)}{\psi(R)}.
\end{equation*}
Recalling $K_{\text{rad}}(R)=-\frac{\psi''(R)}{\psi(R)}$ we get
\begin{equation}\label{Ricci}
    \text{Ric}(\nu,\nu)=(n-1)K_{\mathrm{rad}}(R).
\end{equation}
Substituting \eqref{second fundamental form} and \eqref{Ricci} in \eqref{second variation form} we get
\begin{equation}\label{simplified second variation}
    Q(u,u)=\int_{\Sigma_R}|\nabla^{\Sigma}u|^2 \ \mathrm{d}A-(n-1)\left(K_{\text{rad}}(R)+\left(\frac{\psi'(R)}{\psi(R)}\right)\right)^2\int_{\Sigma_R}u^2\ \mathrm{d}A.
\end{equation}
This simplified form of the second variation formula will be used in the proof of Theorem~\ref{third}.

The isoperimetric problem is strongly connected to the Pol\'ya-Szeg\"o inequality. Following the same lines as in \cite{MV}, we get the following general type inequality
\begin{lemma}\label{L2 PS general}
   Suppose $\m$ satisfies the centered isoperimetric inequality defined in \eqref{Cenetered}. Let $\Omega\subset\m$ be a domain. We define $\Omega^\sharp$ to be the centered (at the pole) geodesic ball with the same volume as $\Omega$. Thern for any $u\in H^1_0(\Omega)$, $u^\sharp\in H^1_0(\Omega^\sharp)$ and 
   \begin{equation*}
       \int_{\Omega^\sharp} |\gradg u^\sharp|^2\dvg\leq \int_{\Omega} |\gradg u|^2\dvg
   \end{equation*}
   holds true.
\end{lemma}
An immediate consequence of Lemma~\ref {L2 PS general} is the Faber-Krahn inequality, which is the link between the isoperimetric problem and the spectrum of the
Laplace-Beltrami operator..
\begin{corollary}[Faber-Krahn inequality]\label{FK inequality}

    Suppose $\m$ satisfies the centered isoperimetric inequality \eqref{Cenetered}. Let $\Omega\subset \m$ be a bounded, measurable set of finite volume. The centered (at the pole) geodesic ball with the same volume as $\Omega$ is denoted by $\Omega^\sharp$. 

    Let $\lambda_1(\Omega)$ and $\lambda_1(\Omega^\sharp)$ denote the first non-zero Dirichlet eigenvalue of the Laplace-Beltrami operator on $\Omega$ and $\Omega^\sharp$, respectively. Then
    \begin{equation*}
        \lambda_1(\Omega) \geq \lambda_1(\Omega^\sharp).
    \end{equation*}
\end{corollary}

\medskip 


The centered isoperimetric inequality provides an effective criterion for determining the discreteness of the spectrum of the Laplace–Beltrami operator. In particular, it allows one to characterize discreteness in terms of a condition depending solely on the underlying warping function $\psi$, making it especially convenient for computational purposes. Following the approach of \cite{CM}, we state the following lemma.
\begin{lemma}\label{discreteness}
    Let $\m$ be an $n(\geq 2)$-dimensional non-compact Cartan-Hadamard manifold of infinite volume, with the warping function $\psi$. Suppose $\m$ satisfies the centered isoperimetric inequality \eqref{Cenetered}. The spectrum of the Laplace-Beltrami operator $\lapg$ on $\m$ is discrete if and only if
    \begin{equation*}
        \lim_{R \to \infty}\left(\int_0^R \psi(t)^{n-1}\dt\right)\left(\int_R^{\infty}\psi(t)^{1-n}\dt\right)=0.
    \end{equation*}
\end{lemma}
\begin{proof}
    By \cite[Theorem~3.1]{CM}, the spectrum is discrete if and only if
    \begin{equation*}
        \lim_{s \to 0}\frac{s}{\mu_{\m}(s)}=0,\quad \lim_{s\to \infty}\frac{s}{\mu_{\m}(s)}=0,
    \end{equation*}
where $\mu_{\m}: [0,\infty)\to [0,\infty)$ is given by
\begin{equation*}
    \mu_M(s)=\inf\left\{C(E): E\,\, \text{is measurable},\,s\leq \text{Vol}_{\m}(E)<\infty\right\}
\end{equation*}
with $C(E)$ being the capacity of $E$. From \cite[Equation~3.7, Section~3]{CM}, the isocapacitary function $\mu_{\m}$ is bounded below by the isoperimetric function $\varrho_{\m}$:
\begin{equation}\label{lower bound}
    \mu_M(s)\geq s+\left(\int_s^{\infty}\frac{\mathrm{d}\xi}{\varrho_{\m}(\xi)^2}\right)^{-1},
\end{equation}
where $\varrho_{\m}: [0,\infty)\to [0,\infty)$ is defined by
\begin{equation*}
    \varrho_M(\xi)=\inf\{\text{Per}_{\m}(E): \xi \leq \text{Vol}_{\m}(E)<\infty\}=\inf_{\xi'\in[\xi,\infty)}\inf\{\text{Per}_{\m}(E):\text{Vol}_{\m}(E)=\xi' \}.
\end{equation*}
By the CII, 
\begin{equation*}
    \inf\{\text{Per}_{\m}(E):\text{Vol}_{\m}(E)=\xi' \}=\text{Per}_{\m}(B_{t'}(0))= \uv\psi(t')^{n-1},
\end{equation*}
where $B_{t'}(0)$ is the unique centered geodesic ball of volume $\xi'$. Since $\xi'\to t'$ is an increasing function and the warping function $\psi$ is increasing as the manifold is negatively curved, we get
\begin{equation*}
    \varrho_{\m}(\xi)=\uv\psi(t)^{n-1},
\end{equation*}
    where $t$ is the unique radius of the geodesic ball centered at the pole having volume $\xi$.

\textbf{Step $1$}: We first check the limit at $0$.
By Taylor expansion near $r=0$, the local geometry is Euclidean. Hence, for small $s=V(R)$, where $V(R)$ stands for the volume of the R radius geodesic ball centered at the pole, we have
\begin{equation*}
    \varrho_{\m}(s)\approx K_n s^{\frac{n-1}{n}},
\end{equation*}
where is a positive constant depending only on $n$. As $s \to 0$,
\begin{equation*}
    \int_s^{\infty}\frac{\mathrm{d}\xi}{\varrho_{\m}(\xi)^2}=\int_s^{s_0}\frac{\mathrm{d}\xi}{\varrho_{\m}(\xi)^2}+\text{non-zero constant}\approx \frac{1}{K_n^2}\int_s^{s_0}\xi^{\frac{2}{n}-2}\mathrm{d}\xi+\text{non-zero constant}.
\end{equation*}
\textbf{Case A}: Let $n\geq 3$. Then
\begin{equation*}
    \int_s^{\infty}\frac{\mathrm{d}\xi}{\varrho_{\m}(\xi)^2}\approx s^{\frac{2-n}{n}}+\text{non-zero constant}.
\end{equation*}
    Hence applying \eqref{lower bound}
    \begin{equation*}
        \mu_M(s)\geq s+s^{\frac{n-2}{2}}.
    \end{equation*}
So,
\begin{equation*}
    0\leq \lim_{s \to 0} \frac{s}{\mu_M(s)}\leq \lim_{s \to 0}\frac{s}{s+s^{\frac{n-2}{n}}}=0.
\end{equation*}
\textbf{Case B}: When $n=2$, then
\begin{equation*}
    \int_s^{\infty} \frac{\mathrm{d}\xi}{\varrho_{\m}(\xi)^2}\approx \int_s^{s_0} \frac{\mathrm{d}\xi}{\xi}+C\approx \ln\left(\frac{1}{s}\right) +C,
\end{equation*}
where $C$ is a non-zero generic constant. Then
\begin{equation*}
   \mu_M(s) \geq s+\frac{1}{\ln\left(\frac{1}{s}\right)+C} 
\end{equation*}
and hence
\begin{equation*}
    0\leq \lim_{s \to 0} \frac{s}{\mu_M(s)}\leq \lim_{s \to 0} \frac{s}{s+\frac{1}{\ln\left(\frac{1}{s}\right)+C}}=0,
\end{equation*}
since $\lim_{s \to 0} s\ln(s)=0$. 
This proves that the first condition $$\lim_{s \to 0}\frac{s}{\mu_{\m}(s)}=0$$ is always satisfied.

\textbf{Step $2$}: Now we consider the limit at infinity. From \eqref{lower bound} it is evident that 
\begin{equation}\label{equivalent}
    \lim_{s \to \infty} s\int_s^{\infty}\frac{\mathrm{d}\xi}{\varrho_M(\xi)^2}=0
\end{equation}
    is a necessary and sufficient condition for the second limit to be zero at infinity. Writing $\xi=V(r)$ and assuming $s=V(R)$ we get
    \begin{equation*}
        \int_s^{\infty} \frac{\mathrm{d}\xi}{\varrho_{\m}(\xi)^2}=\int_R^{\infty}\frac{1}{\left(\uv \psi(r)^{n-1}\right)^2}\left(\uv \psi(r)^{n-1}\right)\dr=\int_R^{\infty}\frac{\dr}{\uv \psi(r)^{n-1}}.
    \end{equation*}
Substituting in \eqref{equivalent} we get,
\begin{equation*}
    \lim_{R \to \infty} V(R)\int_R^{\infty}\frac{\dr}{\uv \psi(r)^{n-1}}=0,
\end{equation*}
that is equivalent to,
\begin{equation*}
    \lim_{R \to \infty}\left( \int_0^R \psi(t)^{n-1}\dt\right)\left(\int_R^{\infty}\psi(t)^{1-n}\dr\right)=0.
\end{equation*}
    In view of the preceding arguments, we arrive at a single necessary and sufficient condition for the Laplace–Beltrami operator to possess a discrete spectrum, namely, that
    \begin{equation*}
        \lim_{R \to \infty}\left( \int_0^R \psi(t)^{n-1}\dt\right)\left(\int_R^{\infty}\psi(t)^{1-n}\dr\right)=0.
    \end{equation*}
\end{proof}

\begin{remark}
  \rm In the hyperbolic space the $\hn$ with $n\geq 2$, all the hypotheses of Lemma~\ref{discreteness} are satisfied and we can compute 
  \begin{equation*}
      \lim_{R \to \infty}\left(\int_0^R (\sinh t)^{n-1}\dt\right)\left(\int_R^{\infty}(\sinh t)^{1-n}\dt\right)=\frac{1}{(n-1)^2} \neq 0
  \end{equation*}
  which implies that the the spectrum of the Laplace-Beltrami operator in $\hn$ is not discrete. Indeed, the spectrum is $\left[\frac{(n-1)^2}{4},\infty\right)$.
  
  For the corresponding result in the finite volume see (\cite[Proposition~2.7]{CM}).\rm
\end{remark}

\medskip

Finally, in the following lemma, we deduce that it is enough to consider the Rayleigh quotient only on the compactly supported smooth \textit{radial} functions to determine the bottom of the spectrum of Laplace-Beltrami operator in the manifold.
\begin{lemma}\label{radial rayleigh quotient}
    Let $\m$ be a warped product manifold. If $\lambda_1(\m)$ denotes the bottom of the spectrum of the Laplace-Beltrami operator $\lapg$ then
    \begin{equation*}
        \lambda_1(\m)=\inf_{u \in C^{\infty}_{c,\mathrm{rad}}(\m)\setminus \{0\}}\frac{\int_{\m}|\gradg u|^2\dvg}{\int_{\m}|u|^2\dvg}.
    \end{equation*}
\end{lemma}
\begin{proof}
Observe that,
\begin{equation*}
    \lambda_1^{\mathrm{rad}}(\m):=\inf_{u \in C^{\infty}_{c,\mathrm{rad}}(\m)\setminus \{0\}}\frac{\int_{\m}|\gradg u|^2\dvg}{\int_{\m}|u|^2\dvg}=\inf_{a\in C_c^{\infty}(0,\infty)\setminus \{0\}}\frac{\int_0^{\infty}a'(r)^2\psi(r)^{n-1}\dr}{\int_0^{\infty}a(r)^2\psi(r)^{n-1}\dr}.
\end{equation*}
   By definition,
    \begin{equation*}
        \lambda_1(\m)=\inf_{u \in C^{\infty}_{c}(\m)\setminus \{0\}}\frac{\int_{\m}|\gradg u|^2\dvg}{\int_{\m}|u|^2\dvg}.
    \end{equation*}
Let $\{Y_{k,m}(\theta)\}$ be a complete orthonormal system of spherical harmonics. For each pair $(k,m)$ and each real-valued compactly supported smooth function $a(r)\in C_c^{\infty}(0,\infty)$ define the $1$D radial quadratic form 
\begin{equation}
    Q_k[a]=\int_0^{\infty}\left(|a'(r)|^2+\frac{\lambda_k}{\psi(r)^2}|a(r)|^2\right)\psi(r)^{n-1}\dr,\quad \|a\|^2_{L^2_r}=\int_0^{\infty}a(r)^2\psi(r)^{n-1}\dr.
\end{equation}
Consequently, for every such non-zero $a\in C^{\infty}_c(0,\infty)$, we have $Q_k[a]\geq Q_0[a]$ and hence
\begin{equation*}
    \frac{Q_k[a]}{\|a\|^2_{L^2_r}}\geq \frac{Q_0[a]}{\|a\|^2_{L^2_r}}\geq \lambda_1^{\text{rad}}(\m).
\end{equation*}
Now recalling \eqref{denominator} and \eqref{gradient simplification with sh},
\begin{equation*}
    \frac{\int_{\m}|\nabla_gu|^2\dvg}{\int_{\m}|u|^2\dvg}=\frac{\sum_{k,m}Q_k[a_{k,m}]}{\sum_{k,m}\|a_{k,m}\|^2_{L^2_r}}\geq \frac{\lambda_1^{\mathrm{rad}}\sum_{k,m}\|a_{k,m}\|^2_{L^2_r}}{\sum_{k,m}\|a_{k,m}\|^2_{L^2_r}}=\lambda_1^{\mathrm{rad}}(\m),
\end{equation*}
which implies $\lambda_1(\m)\geq \lambda_1^{\mathrm{rad}}(\m)$. The other inequality is obvious.

This completes the proof.

\end{proof}

\medskip

\section{Proof of main theorems}\label{main theorem}

\textit{\textbf{Proof of Theorem~\ref{first}}}: Let $\Omega$ be a bounded, measurable set of finite perimeter. By definition,
\begin{align*}
    \text{Per}_{\m}(\Omega)&= \int_{\m}|D\chi_{\Omega}|\dvg\\
    &=\sup\left\{\int_{\Omega}\text{div}X\dvg: X\in C_c^1(M,TM), \|X\|_{\infty} \leq 1\right\}.
\end{align*}
 By Proposition~\ref{approximation}, there exists a sequence of smooth, compactly supported functions $\{u_k\}_{k=1}^\infty \subset C_c^{\infty}(\m)$ such that:
\begin{itemize}
    \item $u_k \to \chi_{\Omega}$ in $L^1(\m)$ as $k\to \infty$.
    \item $\lim_{k\to \infty} \int_{\m}|\gradg u_k|\dvg=\text{Per}_{\m}(\Omega)$.
\end{itemize}
Since Schwarz symmetrization is non-expansive in $L^1$, we get
\begin{equation*}
    \|u_k^\star-\chi_{\Omega}^\star\| \leq \|u_k-\chi_{\Omega}\| \to 0.
\end{equation*}
So, we conclude 
\begin{equation*}
    u_k^\star \to \chi_{\Omega}^\star \quad\text{in}\,\,L^1(\m).
\end{equation*}
The symmetrization $\chi_{\Omega}^\sharp$ is the specific non-increasing function such that its superlevel sets have the same volume as those of $\chi_{\Omega}$. $\chi_{\Omega}$ only takes values $\{0,1\}$:
\begin{itemize}
    \item The set $\{\chi_{\Omega}>t\}$ is $\Omega$ for $t \in (0,1)$.
    \item The set $\{\chi_{\Omega}^\sharp>t\}$ must be a centered geodesic ball $B$ with vol($B$)=vol($\Omega$).
 \end{itemize}
Thus the symmetrization of the characteristic function of $\Omega$ is the characteristic function of the centered (at the pole) geodesic ball $B$, which is of the same volume as that of $\Omega$:\begin{equation*}
    \chi_{\Omega}^\sharp= \chi_B.
\end{equation*}
So, we have established $u_k^\sharp \to \chi_B$ in $L^1(\m)$. By Proposition~\ref{LS}, the total variation (perimeter) functional with respect to the $L^1$ norm is lower semi-continuous. Applying this to our symmetrized function $u_k^{\sharp}$:
\begin{equation*}
    \text{Per}_{\m}(B) =\int_{\m} |D\chi_{B}|\dvg\leq \liminf \int_{\m} |\gradg u_k^\sharp|\dvg.
\end{equation*} 
By the P\'olya-Szeg\"o inequality,
\begin{equation*}
    \int_{\m}|\gradg u_k^\sharp|\dvg \leq \int_{\m} |\gradg u_k|\dvg.
\end{equation*}
Hence,
\begin{equation*}
    \liminf_{k \to \infty} \int_{\m} |\gradg u_k^\sharp|\dvg \leq \liminf_{k \to \infty} \int_{\m} |\gradg u_k|.
\end{equation*}
By our construction, the sequence $u_k$ was chosen specifically so that its gradient energy converges  to the perimeter of $\Omega$:
\begin{equation*}
    \liminf_{k \to \infty} \int_{\m}|\gradg u_k|\dvg =\lim_{k \to \infty} \int_{\m} |\gradg u_k|\dvg= \text{Per}_{\m}(\Omega).
\end{equation*}
Hence,
\begin{equation*}
    \text{Per}_{\m}(B) \leq \liminf_{k \to \infty}\int_{\m}|\gradg u_k^\sharp|\dvg\leq \liminf_{k \to \infty}\int_{\m} |\gradg u_k|\dvg=\text{Per}_{\m}(\Omega).
\end{equation*}
This completes the proof.

\medskip

\textit{\textbf{Proof of Theorem~\ref{second}}}: We recall the mathematical form of the centered isoperimetric inequality and the Cartan-Hadamard conjecture in \eqref{CII} and \eqref{CH-Cojecture} respectively. We write the RHS of \eqref{CII} as $I_{\m}(v)$ and the RHS of \eqref{CH-Cojecture} as $I_{\rn}(v)$ where $v=\text{Vol}(\Omega)$.\par Our claim is :
\begin{equation}\label{pm>pe}
    I_{\m}(v) \geq I_{\rn}(v).
\end{equation}
Once this is proved, we can immediately conclude our theorem.

We define the quotient function $Q(r)$ for a geodesic ball of radius $r$ and centered at the pole in $\m$:
\begin{equation*}
    Q(r)=\frac{\text{Per}_{\m}(B_r(0))^n}{\text{Vol}_{\m}(B_r(0))^{n-1}}.
\end{equation*}
In Euclidean space, where $\psi(r)=r$, this quotient is a constant $c_n=n^n\uv$. We recall that,
\begin{align*}
    &\text{Per}_{\m}(B_r(0))= n|B(0,1)| \psi(r)^{n-1},\\
    &\text{Vol}_{\m}(B_r(0))= n|B(0,1)|\int_0^r\psi(t)^{n-1}\dt
\end{align*}
where $|B(0,1)|$ is the volume of the unit ball in $\rn$. For notational simplicity, we write henceforth $P=P(r)$ for $\text{Per}_{\m}(B_r(0))$ and $V=V(r)$ for $\text{Vol}_{\m}(B_r(0))$.

\textbf{Step $1$}: We differentiate $\ln(Q(r))$ with respect to $r$:
\begin{equation}\label{quo der}
    (\ln Q)'=n\frac{P'}{P}-(n-1)\frac{V'}{V}.
\end{equation}
Using $V'=P$ and substituting $P'= n(n-1)|B(0,1)|\psi(r)^{n-2}\psi(r)$ and simplifying \eqref{quo der} we get
\begin{align}
    (\ln Q)'= n(n-1)\frac{\psi'}{\psi}-(n-1)\frac{P}{V}
    = \frac{n-1}{V\psi}\left[n\psi'V-\psi P\right].
\end{align}

\textbf{Step $2$}: To determine the sign of $(\ln Q)'$ we define bracketed term as our auxiliary function $J(r):= n\psi'V-\psi P$. Obviously, $J(0)=0$. By direct calculation,
\begin{equation*}
    J'(r)= n\psi''V+(n-1)\psi'P-\psi P'.
\end{equation*}
Using $ \psi P'=(n-1)\psi' P$, we simplify
\begin{equation*}
    J'(r)= n\psi''V.
\end{equation*}

\textbf{Step $3$}: Since $\m$ is a Cartan-Hadamard manifold so $\psi'' \geq 0$ for every $r \geq 0$ which, along with the fact that $V$ is always nonnegative, implies that \begin{equation*}
    J'(r) \geq 0.
\end{equation*}
Since $J(0)=0$ and $J'(r) \geq 0$, it follows that $J(r) \geq 0$ for all $r \geq 0$. Consequently, $(\ln Q)' \geq 0$, which means that $\ln Q$ and hence the isoperimetric quotient $Q(r)$ is a non-decreasing function.

\textbf{Step $4$}: Since $\m$ is a Cartan-Hadamard warped product manifold, at the limit $r \to 0$, the geometry is Euclidean. Thus,
\begin{equation*}
    \lim_{r \to 0} Q(r) =Q(0)=Q_{\rn}=n^n|B(0,1)|.
\end{equation*}
Since $Q(r)$ is non-decreasing,
\begin{align*}
    Q(r) \geq Q_{\rn},
\end{align*}
which implies
\begin{equation*}
    \frac{\text{Per}_{\m}(B_r(0)^n}{\text{Vol}_{\m}(B_r(0))^{n-1}} \geq n^n |B(0,1)|.
\end{equation*}
Taking the $n$-th root and rearranging,
\begin{equation*}
    \text{Per}_{\m}(B_r(0) \geq n|B(0,1)|^{\frac{1}{n}}\text{Vol}_{\m}(B_r(0))^{\frac{n-1}{n}},
\end{equation*}
which proves \eqref{pm>pe}.

This completes the proof.

\medskip
 
The following result will be required in the next proof.
\begin{proposition}
    Let $\m$ be a warped product manifold, and $B_R(0)$ be the centered (at the pole) geodesic ball of radius $R$. The $k$-th eigenvalue of the Laplace-Beltrami operator on the boundary of $B_R(0)$ is given by  $\frac{k(k+n-2)}{\psi(R)^2}$
\end{proposition}
\begin{proof}
    The metric of the manifold in spherical coordinates $(r,\theta)$ is given by
    \begin{equation*}
        g=dr^2+\psi(r)^2g_{\sn}.
    \end{equation*}
The boundary of the geodesic ball $B_R(0)$ is the hypersurface $S_R=\{(r,\theta): r=R\}$. The induced metric $h$ on this boundary is obtained by setting $dr=0:$
\begin{equation*}
    h=\psi(R)^2g_{\sn}.
\end{equation*}
    Since the Laplace-Beltrami operator scales inversely with the metric, which can be proved from the coordinate definition of the operator, we get
    \begin{equation}\label{Laplacian change}
        \Delta_{\partial B_R}=\frac{1}{\psi(R)^2}\Delta_{\sn}.
    \end{equation}
    The eigenvalues of the Laplacian on the standard unit sphere $\sn$ are well-known from the theory of spherical harmonics. The eigenvalues $\mu_k$ are given by:
    \begin{equation*}
        \mu_k=k(k+n-2),\,\,k=0,1,2,\cdots.
    \end{equation*}
    If $\lambda_k$ is the $k$-th eigenvalue of $\Delta_{\partial B_R}$, with eigenvector $f$, then using \eqref{Laplacian change} we get
    \begin{align*}
        \Delta_{\partial B_R}f=\lambda_k f,
    \end{align*}
    which implies
    \begin{equation*}
        \Delta_{\sn}f=(\psi(R)^2\lambda_k)f.
    \end{equation*}
This shows that $\lambda_k \psi(R)^2$ must be an eigenvalue of the unit sphere. To find the $k$-th eigenvalue $\lambda_k$ of the boundary, we set the right-hand side equal to the $k$-th eigenvalue $\mu_k$ of the unit sphere:
\begin{align*}
    \lambda_k \psi(R)^2=\mu_k=k(k+n-2)
    \implies\lambda_k=\frac{k(k+n-2)}{\psi(R)^2}
\end{align*}
\end{proof}
\begin{remark}
\rm
    In particular, $\lambda_1=\frac{n-1}{\psi(R)^2}$ and the eigenvector, say $u$, corresponding to $\lambda_1$ satisfies
    \begin{equation*}
        \int_{\Sigma_R}u \mathrm{d}A=0
    \end{equation*}
    since eigenvectors are orthogonal. 
    \rm
\end{remark}

\textit{\textbf{Proof of Theorem~\ref{third}}}: 
\begin{itemize}
    \item[(i)] We prove that if $K_{\mathrm{rad}}$ is increasing then the centered isoperimetric inequality fails. Let $r=\text{dist}_g(x,0)$ for a generic point $x$ in $\m$. From the quotient of $B_{\epsilon}(x)$ we get,
    \begin{align*}
        Q(B_{\epsilon}(x))=\frac{\text{Per}_{\m}(B_{\epsilon}(x))}{\text{Vol}_{\m}(B_{\epsilon}(x))^{\frac{n-1}{n}}}
        \approx \frac{n\uv\epsilon^{n-1}\left(1-\frac{S(r)}{6n}\epsilon^2\right)}{\left[\uv \epsilon^n\left(1-\frac{S(r)}{6(n+2)}\epsilon^2\right)\right]^{\frac{n-1}{n}}}.
    \end{align*}
    Using the Taylor expansion $(1-y)^{\alpha}\approx1-\alpha y$, for small $y$, we get
    \begin{align}\label{first approx}
        Q(B_{\epsilon}(x))&\approx n\uv^{\frac{1}{n}}\left(1-\frac{S(r)}{6n}\epsilon^2\right)\left(1+\frac{n-1}{n}\frac{S(r)}{6(n+2)}\epsilon^2\right)\nonumber\\
        &\approx n\uv^{\frac{1}{n}}\left(1+\left(\frac{n-1}{6n(n+2)}-\frac{1}{6n}\right)S(r)\epsilon^2\right)\nonumber\\
        &=n\uv^{\frac{1}{n}}\left(1-\frac{S(r)}{2n(n+2)}\epsilon^2\right). 
    \end{align}
    Let $V=\text{Vol}_{\m}(B_{\epsilon}(x))=\text{Vol}_{\m}(B_{\epsilon'}(0))$. Now, we derive an estimate on the radius of small geodesic ball in terms of the volume $V$. Starting from the volume expansion for a generic point $x$:
    \begin{align*}
       & V\approx \uv \epsilon^n \left(1-\frac{S(x)}{6(n+2)}\epsilon^2\right),
    \end{align*}
    which implies
    \begin{equation*}
        \epsilon\approx \left(\frac{V}{\uv}\right)^{\frac{1}{n}}\left(1-\frac{S(x)}{6(n+2)}\epsilon^2\right)^{-\frac{1}{n}}.
    \end{equation*}
    Using the Taylor series expansion $(1+y)^{\alpha}\approx(1+\alpha y)$, for small $y$, and taking the first order approximation we get 
    \begin{equation*}
        \epsilon\approx \left(\frac{V}{\uv}\right)^{\frac{1}{n}}.
    \end{equation*}
    Substituting this in \eqref{first approx} we get,
    \begin{equation*}
        Q(B_{\epsilon}(x))\approx n\uv^{\frac{1}{n}}\left(1-\frac{S(r)}{2n(n+2)}\left(\frac{V}{\uv}\right)^{\frac{2}{n}}\right).
    \end{equation*}
To minimize the perimeter for a fixed (small) volume, we must minimize $Q$. From the negative sign in front of the $S(r)$ term, it is evident that minimizing $Q$ is equivalent to maximizing the scalar curvature $S(r)$. This requires to find the point where scalar curvature is closest to $0$ (least negative). If $K_{\text{rad}}(r)$ were increasing then, from lemma~\ref{k to S}, S(r) would also be increasing which implies 
\begin{equation*}
    Q(B_{\epsilon}(x))<Q(B_{\epsilon'}(0))
\end{equation*}
where $\text{Vol}_{\m}(B_{\epsilon}(x))=\text{Vol}_{\m}B_{\epsilon'}(0)$. Hence, a small ball centered far away from the pole has a smaller perimeter than the centered ball of same volume, a contradiction to the centered isoperimetric inequality. \par This proves that $K_{\text{rad}}(r)$ can not be increasing in $(0,\infty)$ as a function of $r$.
\item[(ii)] In the equation \eqref{simplified second variation}, we put $u$ to be the eigenfunction corresponding to the first non-zero eigenvalue $\lambda_1$ of the Laplace-Beltrami operator on the geodesic sphere $\Sigma_R$.
Using Green's identity and the eigenvalue equation 
\begin{equation*}
    \Delta_{\Sigma_R}u=-\lambda_1 u,
\end{equation*}
we get,
\begin{equation*}
    \int_{\Sigma_R}|\nabla^{\Sigma}u|^2\ \mathrm{d}A=-\int_{\Sigma_R} u(\Delta_{\Sigma}u)\ \mathrm{d}A=-\int_{\Sigma_R}u(-\lambda_1 u)\ \mathrm{d}A=\lambda_1\int_{\Sigma_R}u^2\ \mathrm{d}A=\frac{n-1}{\psi(R)^2}\int_{\Sigma_R}u^2\ \mathrm{d}A.
\end{equation*}
Now, substituting this gradient term in \eqref{simplified second variation} and using the stability condition we get
\begin{equation*}
    \frac{n-1}{\psi(R)^2}\int_{\Sigma_R}u^2\ \mathrm{d}A-(n-1)\left(K_{\text{rad}}(R)+\left(\frac{\psi'(R)}{\psi(R)}\right)^2\right)\int_{\Sigma_R}u^2\ \mathrm{d}A \geq 0,
\end{equation*}
which shows,
\begin{equation*}
    \frac{1}{\psi(R)^2}-\left(K_{\text{rad}}(R)+\left(\frac{\psi'(R)}{\psi(R)}\right)^2\right) \geq 0,
\end{equation*}
and this implies
\begin{equation*}
    K_{\text{rad}}(R) \leq K_{\text{tan}}(R),\,\,\,\text{for every $R>0$}.
\end{equation*}
\item[(iii)] Since we have already deduced
\begin{equation*}
    K_{\text{tan}}'(r)=2\frac{\psi'(r)}{\psi(r)}(K_{\text{rad}}(r)-K_{\text{tan}}(r)),
\end{equation*}
the conclusion immediately follows from $(ii)$ and recalling that $\frac{\psi'(r)}{\psi(r)}>0$ for any $r>0$.
\end{itemize}

\medskip 

\begin{remark}
{\rm 
In view of the computation in part~(i), one deduces that, for a fixed sufficiently small volume, a geodesic ball in a ``flatter'' space (that is, a space whose scalar curvature is less negative) has a smaller perimeter than a geodesic ball of the same volume in a space with more negative curvature. As a consequence, \emph{the Cartan-Hadamard conjecture holds for sets of sufficiently small volume}.}
\end{remark}

\medskip 

Below, we define the \textit{isoperimetric quotient }
\begin{equation}\label{isoperimetric quotient}
    I(r)=\frac{\text{Per}_{\m}(B_r(0))}{\text{Vol}_{\m}(B_r(0))}.
\end{equation}
The following lemma can be compared with the Bishop-Gromov comparison theorem \cite[Theorem~1.3]{AMR}.
\begin{lemma}
   Let $\m$ be a Cartan-Hadamard warped product manifold. The function defined by 
    \begin{equation*}
        f(r)=\frac{\text{Vol}_{\m}(B_r(0))}{r^n} \quad \forall\,r>0
    \end{equation*}
    is monotone increasing and hence $$I(r) \geq \frac{n}{r}\,\,\,\, \forall \, r>0.$$
    Furthermore, if the radial sectional curvature $K_{\text{rad}}(r)$ is increasing, then $I(r)$ is decreasing.
\end{lemma}
\begin{proof}
    We recall,
    \begin{align*}
        &V(r)=\text{Vol}_{\m}(B_r(0))= \uv \int_0^r \psi(t)^{n-1}\dt,\\
        &A(r)= \text{Per}_{\m}(B_r(0))=\uv \psi(r)^{n-1}.
    \end{align*}
    Hence, $V'(r)=A(r)$.
    By direct calculation, it follows that,
    \begin{equation*}
        f'(r)=\frac{rA(r)-nV(r)}{r^{n+1}}.
    \end{equation*}
    If we define $g(r)=rA(r)-nV(r)$ the obviously $g(0)=0$. Also,
    \begin{equation*}
        g'(r)=\uv (n-1)\psi(r)^{n-2}[r\psi'(r)-\psi(r)].
    \end{equation*}
Since $\psi''(r) \geq 0$ the function $\psi'(r)$ is non-decreasing. Hence, using $\psi(0)=0$,
\begin{align*}
    &\psi (r)=\int_0^r \psi'(t)\dt\leq \int_0^r\psi'(r)\dt=r\psi'(r),
\end{align*}
which proves that $g'(r) \geq 0$ and hence $g(r)$ is nonnegative on $(0,\infty)$. This proves 
\begin{equation*}
    I(r) \geq \frac{n}{r}
\end{equation*}
and consequently, the monotone increasingness of $f$.\par 
For the second part, we write 
\begin{align*}
    &S(r)=\psi(r)^{n-1},\\
    &G(r)=\int_{0}^r\psi(t)^{n-1}\dt.
\end{align*}
Then using the fact $G'(r)=S(r)$,
\begin{equation}\label{isoperimetric quotient derivative}
    I'(r)=\frac{S'(r)G(r)-(S(r))^2}{(G(r))^2}.
\end{equation}
Observe that,
\begin{align}\label{simplification}
    S'(r)G(r)-(S(r))^2=& \int_0^r(S'(r)S(t)-S(r)S'(t))\dt\nonumber\\
    =&\int_0^r S(t)S(r)\left(\frac{S'(r)}{S(r)}-\frac{S'(t)}{S(t)}\right)\dt.
\end{align}
Since 
\begin{equation*}
    \frac{S'(r)}{S(r)}=H(r),\quad \forall r>0
\end{equation*}
and $H(r)$ is decreasing as $K_{\text{rad}}(r)$ is increasing, by lemma~\ref{radial to mean}, we conclude from \eqref{simplification} 
\begin{equation*}
    S'(r)G(r)-(S(r))^2\leq 0.
\end{equation*}
This proves from \eqref{isoperimetric quotient derivative} that $I'(r)\leq 0$ and hence $I(r)$ is monotone decreasing.
\end{proof}
\begin{remark}
\rm
    If $I(r)$ is decreasing then
    \begin{equation*}
        L:=\lim_{r \to \infty}I(r)
    \end{equation*}
    exists finitely and 
    \begin{equation}\label{comparison with Cheeger constant}
        L \geq h(\m),
    \end{equation}
    where $h(\m)$ is the Cheeger constant of the manifold $\m$. If the centered isoperimetric inequality holds, then equality occurs in \eqref{comparison with Cheeger constant}. \rm
\end{remark}

\medskip

Next, we wish to derive an explicit form of the first eigenvalue of a class of manifolds and thereby improve the Cheeger inequality.

   \textit{\textbf{Proof of Theorem~\ref{fourth}}}: By L'Hospital's rule,
    \begin{equation*}
        L=\lim_{r \to \infty}I(r)=\lim_{r \to \infty}\frac{\text{Per}_{\m}(B_r(0))}{\text{Vol}_{\m}(B_r(0))}=(n-1)\lim_{r \to \infty}\frac{\psi'(r)}{\psi(r)}.
    \end{equation*}
Using lemma~\ref{radial rayleigh quotient} we write
\begin{equation}\label{radial eigenvalue}
    \lambda_1(\m)=\inf_{u\in C_c^{\infty}(0,\infty)\setminus \{0\}}\frac{\int_0^{\infty}(u'(r)^2\psi(r)^{n-1}\dr)}{\int_0^{\infty}(u(r))^2\psi(r)^{n-1}\dr}.
\end{equation}
Consider the transformation
\begin{equation*}
    v(r)=u(r)\sqrt{S(r)},\quad \text{where $S(r)=\psi(r)^{n-1}$}.
\end{equation*}
Hence,
\begin{equation*}
    (u')^2S=\left(\frac{v'}{\sqrt{S}}-\frac{vS'}{2S^{\frac{3}{2}}}\right)^2S=(v')^2-vv'\frac{S'}{S}+\frac{1}{4}v^2\left(\frac{S'}{S}\right)^2.
\end{equation*}
Integrating the middle terms by parts, observing that $v$ has compact support
\begin{equation*}
    -\int_0^{\infty}v(r)v'(r)\frac{S'(r)}{S(r)}\dr=-\int_0^{\infty}\frac{1}{2}(v(r)^2)'\frac{S'(r)}{S(r)}\dr=\frac{1}{2}\int_0^{\infty}(v(r)^2)\left(\frac{S'(r)}{S(r)}\right)'\dr.
\end{equation*}
    Now, assembling the complete energy integral in terms of $v$,
    \begin{equation*}
        \int_0^{\infty}(u'(r))^2S(r)\dr=\int_0^{\infty}\left[(v'(r))^2+\left(\frac{1}{4}\left(\frac{S'(r)}{S(r)}\right)^2+\frac{1}{2}\left(\frac{S'(r)}{S(r)}\right)'\right)v(r)^2\right]\dr.
    \end{equation*}
The denominator of \eqref{radial eigenvalue} becomes simply $\int_0^{\infty}u(r)^2S(r)\dr=\int_0^{\infty}v(r)^2\dr$.

Thus $\lambda_1(\m)$ is the bottom of the spectrum of the Schr\"odinger operator $H:=-\frac{d^2}{dr^2}+W(r)$ where the potential $W(r)$ is
\begin{equation*}
    W(r)=\frac{1}{4}\left(\frac{S'(r)}{S(r)}\right)^2+\frac{1}{2}\left(\frac{S'(r)}{S(r)}\right)'.
\end{equation*}
    By assumptions,
    \begin{equation*}
        \lim_{r \to \infty}W(r)=\frac{L^2}{4}.
    \end{equation*}
    According to Persson's theorem, for a potential $W(r)$ that converges to a limit at infinity, the bottom of the essential spectrum is exactly that limit:
    \begin{equation*}
        \lambda_1(\m)= \frac{L^2}{4}.
    \end{equation*}
    This completes the proof.

\begin{remark}
\rm
   In \cite[Lemma~4.1]{BFG}, it is shown that if $L>0$ then $\lambda_1(\m)>0$. Theorem~\ref{fourth} strengthens this conclusion by providing an explicit form of the first non-zero eigenvalue. As $L \geq h(\m)$, where $h(\m)$ is the Cheeger constant of the manifold, this clearly improves the Cheeger inequality:
    \begin{equation*}
        \lambda_1(\m)\geq \frac{h(\m)^2}{4}.
    \end{equation*}
    For the hyperbolic space $\hn$, where $\psi(r)=\sinh r$, $L=(n-1)$ and hence $\lambda_1(\hn)=\frac{(n-1)^2}{4}$.\rm
\end{remark}
\begin{remark}\label{strict improvement}
\rm
    For the warped product manifold $\mathbb{M}^2$ with the warping function
    \begin{equation*}
        \psi(r)=\sinh r+Ar^2e^{\frac{r}{2}},
    \end{equation*}
    where $A>0$ is a large constant, the constant $L$, as defined in theorem~\ref{fourth}, is strictly greater than the Cheeger constant $h(\mathbb{M}^2)$. 

    We calculate,
    \begin{equation*}
        L=\lim_{r \to \infty}\frac{\psi'(r)}{\psi(r)}=\lim_{r \to \infty}\frac{\left(1+e^{-2r}\right)+2A\left(2r+\frac{1}{2}r^2\right)e^{-{\frac{r}{2}}}}{\left(1-e^{-2r}\right)+2Ar^2e^{-{\frac{r}{2}}}}=1.
    \end{equation*}
   Similarly we can show,
   \begin{equation*}
       \lim_{r \to \infty}\left(\frac{\psi'(r)}{\psi(r)}\right)'=0.
   \end{equation*}
Now we show that there exists an $r$ where $I(r)<1$. This proves our claim since 
\begin{equation*}
    h(\mathbb{M}^2)\leq I(r),\quad \forall\, r>0.
\end{equation*}
    Observe that,
    \begin{equation*}
        \int_0^r\psi(t)\dt=(\cosh (r)-1)+A\left[(2r^2-8r+16)e^{\frac{r}{2}}-16\right].
    \end{equation*}
    Now, we evaluate $I(r)$ at a fixed large $r$, say $r=10$:
    \begin{equation*}
        I(10)=\frac{\sinh (10)+100Ae^5}{(\cosh (10)-1)+A(136e^5-16)}.
    \end{equation*}
    For sufficiently large $A$, $I(10)\approx 0.7356<1$
    \rm
\end{remark}

\medskip

\begin{remark}
\rm
    If the manifold $\m$ has a spectral gap, then the manifold is necessarily non-parabolic. Indeed, if it satisfies the Poincaré inequality with best constant $\lambda_1(\m)$, then restricting to the radial functions, we can write for every $f \in C_c^{\infty}(\mathbb{R}\setminus \{0\})$
    \begin{equation}\label{1D spectral}
        \int_0^{\infty} (f'(r))^2\psi(r)^{n-1}\dr \geq \lambda_1(\m)\int_0^{\infty} (f(r))^2\psi(r)^{n-1}\dr.
    \end{equation}
    Using \cite[Theorem~2, section~1.3, Chapter~1]{Mazja}, \eqref{1D spectral} holds if and only if
    \begin{equation*}
        \sup_{r>0}\left[\int_0^r \psi(t)^{n-1}\dt\int_r^{\infty}\psi(t)^{1-n}\dt\right]<\infty.
    \end{equation*}
    This immediately implies,
    \begin{equation*}
        \int_1^{\infty}\psi(t)^{1-n}\dt<\infty,
    \end{equation*}
    which ensures non-parabolicity of the manifold (see \cite[Proposition~3.1]{grigo}). This gives the existence of a positive, minimal Green's function. The converse is not true; the Euclidean space $\mathbb{R}^n$ with $n \geq 3$ provides a counterexample.\rm
\end{remark}

\medskip

Now we prove an auxiliary result that will be required in our further analysis.
\begin{proposition}\label{vol upper bund}
    Let $\m$ be a Riemannian manifold with warping function $\psi$ and Riemannian distance $\varrho$. Fix $r>0$ and suppose the function 
    \begin{equation*}
        g(t):= \frac{t\psi'(t)}{\psi(t)}
    \end{equation*}
    is nondecreasing on $[0,r]$. Then for every measurable set $A\subset B_r(0)$
    \begin{equation*}
        \int_A\frac{\varrho(x)\psi'(\varrho(x))}{\psi(\varrho(x))}\dvg\geq \mathrm{Vol}_{M}(A),
    \end{equation*}
    where $\varrho(\cdot)$ is the geodesic distance from the fixed pole.
\end{proposition}
\begin{proof}
    We write 
    \begin{equation*}
        V_A(t):=\text{Vol}_{M}(A\cap B_t(0))= \int_0^t\left(\int_{A \cap \partial B_s}\mathrm{d}\sigma_s\right)\ \mathrm{d}s\quad \forall\, t\in [0,r].
    \end{equation*}
So, for a.e. $t$, $$V_A'(t)=\int_{A \cap \partial B_t} \ \mathrm{d}\sigma _t.$$
The coarea formula gives, for any nonnegative measurable function $F(\varrho)$,
\begin{equation*}
    \int_A F(\varrho(x))\dvg= \int_0^r F(t)\left(\int_{A \cap \partial B_t} \ \mathrm{d}\sigma _t\right)\dt=\int_0^r F(t)V_A'(t)\dt,
\end{equation*}
for almost every $t$. Applying this with $g$ we deduce,
\begin{equation*}
    \int_A \frac{\varrho \psi'(\varrho)}{\psi(\varrho)}\dvg=\int_0^r g(t)V_A'(t)\dt.
\end{equation*}
Hence, it is enough to prove 
\begin{equation}\label{first reduction}
    \int_0^r g(t)V_A'(t)\dt \geq V_A(r).
\end{equation}. 
The functions $t \to g(t)$ and $t \to V_A(t)$ are locally integrable, and $V_A(t)$ is absolutely continuous. By integration by parts and noting $g(0)=1$,
\begin{equation*}
    \int_0^r g(t)V_A'(t)\dt =g(r)V_A(r)-\int_0^rg'(t)V_A(t)\dt.
\end{equation*}
Hence \eqref{first reduction} is equivalent to
\begin{equation}\label{second reduction}
    \int_0^r g'(t)V_A(t)\dt \leq (g(r)-1)V_A(r).
\end{equation}
Because $g$ is assumed monotone increasing, we have $g'(t) \geq 0$ a.e. . Hence
\begin{equation*}
    g'(t)V_A(t) \leq g'(t)V_A(r)\quad \forall \, t\in [0,r].
\end{equation*}
Integrating and using $g(0)=1$ we get
\begin{equation*}
    \int_0^r g'(t)V_A(t)\dt \leq V_A(r)(g(r)-1).
\end{equation*}
Finally, we get
\begin{equation*}
    \int_A \frac{\varrho \psi'(\varrho)}{\psi(\varrho)}\dvg=g(r)V_A(r)-\int_0^r g'(t)V_A(t)\dt\geq g(r)V_A(r)-(g(r)-1)V_A(r)=V_A(r).
\end{equation*}
\end{proof}

\textit{\textbf{Proof of Theorem~\ref{fifth}}}: Define the smooth radial vector field on $B_r(0)$ by
\begin{equation*}
    X(x)=\frac{\varrho(x)}{r}\partial _\varrho,
\end{equation*}
where $\varrho(x)=\text{dist}_g(0,x)$ and $\partial_ \varrho$ is the unit radial vector. Since
\begin{equation*}
    \text{div}(\partial_\varrho)= (n-1)\frac{\psi'(\varrho)}{\psi(\varrho)},
\end{equation*}
we immediately get, using $\text{div}(fX)=X(f)+f\text{div}(X)$ for any smooth function $f$,
\begin{equation}\label{div}
    \text{div}(X)= \frac{1}{r}+\frac{\varrho}{r}(n-1)\frac{\psi'(\varrho)}{\psi(\varrho)}.
\end{equation}
Let $A\subset B_r(0)$ be any measurable set of finite perimeter. From divergence theorem we get,
\begin{equation*}
    \int_{A}\text{div}(X)\dvg=\int_{\partial A}\langle X,\nu_A\rangle\ d\mathcal{H}^{n-1}\leq \int_{\partial A}|X|\ d\mathcal{H}^{n-1}\leq \text{Per}_{\m}(A),
\end{equation*}
since $|X|\leq 1$ on $B_r(0)$. Using \eqref{div} we get
\begin{equation*}
    \int_{A}\text{div}(X)\dvg= \frac{1}{r}\text{Vol}_{M}(A)+\frac{n-1}{r}\int_{A}\frac{\varrho \psi'(\varrho)}{\psi(\varrho)}\dvg.
\end{equation*}
Thus,
\begin{equation*}
    \text{Per}_{\m}(A) \geq \frac{1}{r}\text{Vol}_{M}(A)+\frac{n-1}{r}\int_{A}\frac{\varrho \psi'(\varrho)}{\psi(\varrho)}\dvg.
\end{equation*}
Applying proposition~\ref{vol upper bund} we deduce
\begin{equation}\label{refinement}
    \text{Per}_{\m}(A) \geq \frac{1}{r}\text{Vol}_{M}(A)+\frac{n-1}{r}\text{Vol}_{M}(A)=\frac{n}{r}\text{Vol}_{M}(A).
\end{equation}
By definition of the Cheeger constant $B_r(0)$, denoted by $h(B_r)$,
\begin{equation*}
    h(B_r)= \inf_{A\subset B_r}\frac{\text{Per}_{M}(A)}{\text{Vol}_{M}(A)}.
\end{equation*}
By \eqref{refinement},
\begin{equation*}
    h(B_r)\geq \frac{n}{r}.
\end{equation*}
Finally, we conclude our theorem by invoking the Cheeger inequality,
\begin{equation*}
    \lambda_1(B_r)\geq \frac{h(B_r)^2}{4}\geq \frac{n^2}{4r^2}.
\end{equation*}

This completes the proof.

\begin{remark}\label{final remark}
\rm
    It is easy to see that $g(t):=\frac{t\psi'(t)}{\psi(t)}$ is increasing if and only if there exist constants $A>0$, $C\in \mathbb{R}$ and a function
    \begin{equation*}
        \phi\in L^1_{\text{loc}}(0,\infty),\quad \phi\geq 0,
    \end{equation*}
    such that
    \begin{equation*}
        \psi(t)=At^C\exp\left(\int_{t_0}^t\frac{1}{\tau}\left(\int_{t_0}^\tau \phi(s)\ds\right)\ \mathrm{d}\tau\right),
    \end{equation*}
    for any fixed $t_0>0$. Indeed, if we define 
    \begin{equation*}
        h(t)=\log(\psi(t))
    \end{equation*}
    then 
    \begin{equation*}
        g'(t)=(th'(t))'.
    \end{equation*}
    Hence it is equivalent to characterize when the function $th'(t)$ is monotone increasing. By applying fundamental theorem of calculus, this is the case if and only if there exists $C\in \mathbb{R}$ and a nonnegative $\phi\in L^1_{\text{{loc}}}(0,\infty)$ such that
    \begin{equation*}
        th'(t)=C+\int_{t_0}^t\phi(s)\ds,
    \end{equation*}
    for any fixed $t_0>0$. Simplifying,
    \begin{equation*}
        \psi(t)=At^C\exp\left(\int_{t_0}^t\frac{1}{\tau}\left(\int_{t_0}^\tau \phi(s)\ds\right)\ \mathrm{d}\tau\right),
    \end{equation*}
    where $A=\frac{\exp (h(t_0))}{t_0^C}>0$. 
    
    Some particular cases where $g$ is monotone increasing are $\psi(t)=t^C$ for any $C\in \mathbb{R}$, $\psi(t)=\sinh t$ and $\psi(t)=t^C\exp(\alpha t^2)$ for any $C \in \mathbb{R}$ and $\alpha>0$.
    \rm
\end{remark}

\section*{Acknowledgments} 
The author is grateful to Prof. Debdip Ganguly for the useful discussion. The research is supported by the Doctoral Fellowship of the Indian Statistical Institute, Delhi Centre.


\noindent 
	
\begin{thebibliography}{100}
\bibitem{AMR} L.~J. ~Al\'ias, P.~Mastrolia, M.~Rigoli,  \textit{Maximum Principles and Geometric Applications}, Springer Monogr. Math., Springer, Cham, (2016), MR\href{https://doi.org/10.1007/978-3-319-24337-5}{3445380}.

\bibitem{Aubin} T.~Aubin, \emph{ Probl\`emes isop\'erim\'etriques et espaces de Sobolev}, J. Differential Geometry 11 (1976), no. 4, 573–598, MR\href{http://projecteuclid.org/euclid.jdg/1214433725}{448404}

\bibitem{bnstn} A.~Baernstein~II, \textit{Symmetrization in analysis}, Volume 36 of New Mathematical Monographs. Cambridge University Press, Cambridge. With David Drasin and Richard S. Laugesen, with a foreword by Walter Hayman (2019),  MR\href{https://doi.org/10.1017/9781139020244}{3929712}.

\bibitem{BR} E.F.~ Beckenbach, T.~Rad\'o, \emph{Subharmonic functions and surfaces of negative curvature}, Trans. Amer. Math. Soc. 35 (1933), no. 3, 662–674, MR\href{https://doi.org/10.2307/1989854}{1501708}



\bibitem{CB} I.~Benjamini, J.~Cao, \emph{A new isoperimetric comparison theorem for surfaces of variable curvature}, Duke Math. J. 85 (1996), no. 2, 359–396, MR\href{https://doi.org/10.1215/S0012-7094-96-08515-4}{1417620}



\bibitem{BS} C.~Bennett, S.~Robert, \emph{Interpolation of operators}, Pure Appl. Math., 129, Academic Press, (1988), MR\href{https://mathscinet.ams.org/mathscinet/article?mr=928802}{0928802}.

\bibitem{BFG} E.~Berchio, A.~Ferrero, G.~Grillo, \emph{Stability and qualitative properties of radial solutions of the Lane–Emden–Fowler equation on Riemannian models} J. Math. Pures Appl. (9) 102 (2014), no. 1, 1–35, MR\href{https://doi.org/10.1016/j.matpur.2013.10.012}{3212246}

\bibitem{ps-hn} V.~Bögelein, F.~Duzaar, C.~Scheven, \textit{A sharp quantitative isoperimetric inequality in hyperbolic $n$-space}, Calc. Var. Partial Differential Equations 54 (2015), 3967–4017, MR\href{https://doi.org/10.1007/s00526-015-0928-9}{3426101}.



\bibitem{Brendle} S.~Brendle, \emph{ Constant mean curvature surfaces in warped product manifolds}, Publ. Math. Inst. Hautes Études Sci. 117 (2013), 247–269, MR\href{https://doi.org/10.1007/s10240-012-0047-5}{3090261}

\bibitem{Brooks 1} R.~Brooks, \emph{The fundamental group and the spectrum of the Laplacian}, Comment. Math. Helv. 56 (1981), no. 4, 581–598, MR\href{https://doi.org/10.1007/BF02566228}{656213}



\bibitem{Brooks 2} R.~Brooks, \emph{A relation between growth and the spectrum of the Laplacian}, Math. Z. 178 (1981), no. 4, 501–508, MR\href{https://doi.org/10.1007/BF01174771}{638814}



\bibitem{Buser} P.~Buser, \emph{Über eine ungleichung von cheeger}, Math. Z. 158 (1978), no. 3, 245–252, MR\href{https://doi.org/10.1007/BF01214795}{478248}





\bibitem{Chavel} I.~Chavel, \emph{Eigenvalues in Riemannian geometry}, Pure Appl. Math., 115
Academic Press, Inc., Orlando, FL, 1984. xiv+362 pp.
ISBN:0-12-170640-0, MR{768584}

\bibitem{Cheeger} J.~Cheeger, \emph{A lower bound for the smallest eigenvalue of the Laplacian}, Problems in analysis, pp. 195–199, Princeton University Press, Princeton, NJ, 1970, MR{402831}

\bibitem{CM} A.~Cianchi, V.~Maz'ya, \emph{On the discreteness of the spectrum of the laplacian on noncompact Riemannian manifolds}, J. Differential Geom. 87 (2011), no. 3, 469–491, MR\href{http://projecteuclid.org/euclid.jdg/1312998232}{2819545}

\bibitem{Croke} C.B.~Croke, \emph{A sharp four dimensional isoperimetric inequality}, Comment. Math. Helv. 59 (1984), no. 2, 187–192, MR\href{https://doi.org/10.1007/BF02566344}{749103}



\bibitem{Dodziuk 1} J.~Dodziuk, \emph{Every covering of a compact Riemann surface of genus greater than one carries a nontrivial $L^2$ harmonic differential}, Acta Math. 152 (1984), no. 1-2, 49–56, MR\href{https://doi.org/10.1007/BF02392190}{736211}



\bibitem{Dodziuk 2} J.~Dodziuk, \emph{Difference equations, isoperimetric inequality and transience of certain random walks}, Trans. Amer. Math. Soc. 284 (1984), no. 2, 787–794, MR\href{https://doi.org/10.2307/1999107}{743744}



\bibitem{Donnely} H.~Donnely, \emph{On the essential spectrum of a complete Riemannian manifold}, Topology 20 (1981), no. 1, 1–14, MR\href{https://doi.org/10.1016/0040-9383(81)90012-4}{592568}



\bibitem{DHV} O.~Druet, E.~Hebey, M.~Vaugon, \emph{ Optimal Nash’s inequalities on Riemannian manifolds: the
influence of geometry}, Internat. Math. Res. Notices 1999, no. 14, 735–779 MR\href{https://doi.org/10.1155/S1073792899000380}{1704184}


\bibitem{Fusco} N.~Fusco, \emph{The quantitative isoperimetric inequality and related topics}, Bull. Math. Sci. 5 (2015), no. 3, 517–607, MR\href{https://doi.org/10.1007/s13373-015-0074-x}{3404715}



\bibitem{FMP} N.~ Fusco, F.~Maggi, A.~Pratelli, \emph{The sharp quantitative isoperimetric inequality}, Ann. of Math. (2) 168 (2008), no. 3, 941–980, MR\href{https://doi.org/10.4007/annals.2008.168.941}{2456887}

\bibitem{Giusti} E.~Giusti, \emph{Minimal Surfaces and Functions of Bounded Variation}, Monogr. Math., 80
Birkhäuser Verlag, Basel, 1984. xii+240 pp.
ISBN:0-8176-3153-4, MR\href{https://doi.org/10.1007/978-1-4684-9486-0}{775682}

\bibitem{GW} R.~E.~Greene, H.~Wu, \emph{Function Theory on Manifolds which Possess a Pole}, Lecture Notes in Math., 699, Springer, Berlin, (1979), MR\href{https://doi.org/10.1007/BFb0063413}{0521983}.

\bibitem{grigo} A.~Grigor'yan, \emph{Analytic and geometric background of recurrence and non-explosion of the Brownian motion on Riemannian manifolds}, Bull. Amer. Math. Soc. (N.S.) 36 (1999), no. 2, 135–249 MR\href{https://doi.org/10.1090/S0273-0979-99-00776-4}{1659871}

\bibitem{grig}  A.~Grigor’yan, \emph{Heat Kernel and Analysis on Manifolds}, AMS/IP Stud. Adv. Math., 47, American Mathematical Society, Providence, RI; International Press, Boston, MA, (2009), MR\href{https://doi.org/10.1090/amsip/047}{2569498}.

\bibitem{ps-ball} M.~Gromov, \textit{Metric structures for Riemannian and non-Riemannian spaces}, Mod. Birkhäuser Class, Birkhäuser Boston, Inc., Boston, MA, (2007), MR\href{https://mathscinet.ams.org/mathscinet/article?mr=682063}{2307192}.

\bibitem{Hebey} E.~Hebey, \emph{Nonlinear analysis on manifolds: Sobolev spaces and inequalities}, Courant Lect. Notes Math., 5 New York University, Courant Institute of Mathematical Sciences, New York; American Mathematical Society, Providence, RI, 1999. x+309 pp MR{1688256}

\bibitem{Kleiner} B.~Kleiner, \emph{An isoperimetric comparison theorem}, Invent. Math. 108 (1992), no. 1, 37–47, MR\href{https://doi.org/10.1007/BF02100598}{1156385}



\bibitem{KK} B.R.~ Kloeckner, G.~Kuperberg, \emph{The Cartan–Hadamard conjecture and the little prince}, Rev. Mat. Iberoam. 35 (2019), no. 4, 1195–1258, MR\href{https://doi.org/10.4171/rmi/1082}{3988083}



\bibitem{AK} A.~Krist\'aly, \emph{Sharp Functional Inequalities and Elliptic Problems on Non-Euclidean Structures}, Éles funkcionál-egyenlőtlenségek és elliptikus problémák nemeuklideszi struktúrákon. Diss. Óbudai Egyetem, 2017.

\bibitem{Maggi} F.~Maggi, \emph{ Sets of Finite Perimeter and Geometric Variational Problems}, An introduction to geometric measure theory
Cambridge Stud. Adv. Math., 135
Cambridge University Press, Cambridge, 2012. xx+454 pp.
ISBN:978-1-107-02103-7, MR\href{https://doi.org/10.1017/CBO9781139108133}{2976521}

\bibitem{Mazja} V.~Maz'ja, \emph{Sobolev spaces}, Springer Ser. Soviet Math.
Springer-Verlag, Berlin, 1985. xix+486 pp.
ISBN:3-540-13589-8, MR\href{https://doi.org/10.1007/978-3-662-09922-3}{817985}

\bibitem{McKean} H.P.~McKeam, \emph{An upper bound to the spectrum of $\Delta$ on
a manifold of negative curvature}, J. Differential Geometry 4 (1970), 359–366, MR\href{http://projecteuclid.org/euclid.jdg/1214429509}{266100}

\bibitem{Miranda} M.~Miranda, Jr., \emph{ Functions of bounded variation on “good” metric spaces}, J. Math. Pures Appl. (9) 82 (2003), no. 8, 975–1004, MR\href{https://doi.org/10.1016/S0021-7824(03)00036-9}{2005202}

\bibitem{MPPP} M.~Miranda, Jr., D.~Pallara, F.~Paronetto, M.~Preunkert \emph{Heat semigroup and functions of 
bounded variation on Riemannian manifolds}, J. Reine Angew. Math. 613 (2007), 99–119, MR\href{https://doi.org/10.1515/CRELLE.2007.093}{2377131}

\bibitem{Morgan} F.~Morgan, \emph{ Regularity of isoperimetric hypersurfaces in Riemannian manifolds}, Trans. Amer. Math. Soc. 355 (2003), no. 12, 5041–5052, MR\href{https://doi.org/10.1090/S0002-9947-03-03061-7}{1997594}

\bibitem{MJ} F.~Morgan, D.L.~Johnson, \emph{Some sharp isoperimetric theorems for Riemannian manifold}, Indiana Univ. Math. J. 49 (2000), no. 3, 1017–1041, MR\href{https://doi.org/10.1512/iumj.2000.49.1929}{1803220}



\bibitem{FMH} F.~Morgan, M~Hutchings, H~Howards, \emph{The isoperimetric problem on surfaces of revolution of decreasing Gauss curvature}, Trans. Amer. Math. Soc. 352 (2000), no. 11, 4889–4909, MR\href{https://doi.org/10.1090/S0002-9947-00-02482-X}{1661278}

\bibitem{MV} M~Muratori, B.~Volzone, \emph{Concentration comparison for nonlinear diffusion on model manifolds and P\'olya-Szeg\H {o} inequality}, arXiv preprint arXiv:2507.19279 (2025).

\bibitem{Muratori-Soave} M.~Muratori, N.~Soave, \emph{Some rigidity results for Sobolev inequalities and related PDEs on Cartan-Hadamard manifolds}, Ann. Sc. Norm. Super. Pisa Cl. Sci. (5) 24 (2023), no. 2, 751–792, MR\href{https://doi.org/10.2422/2036-2145.202105_071}{4630043}

\bibitem{Osserman} R.~Osserman, \emph{The isoperimetric inequality}, Bull. Amer. Math. Soc. 84 (1978), no. 6, 1182–1238, MR\href{https://doi.org/10.1090/S0002-9904-1978-14553-4}{500557}



\bibitem{Osser} R.~Osserman, \emph{A note on Hayman's theorem on the bass note of a drum}, Comment. Math. Helv. 52 (1977), no. 4, 545–555, MR\href{https://doi.org/10.1007/BF02567388}{459099}



\bibitem{Pansu} P.~Pansu, \emph{Sur la régularité du profil isopérimétrique des surfaces riemanniennes compactes}, Ann. Inst. Fourier (Grenoble) 48 (1998), no. 1, 247–264, MR\href{https://doi.org/10.5802/aif.1617}{1614957}



\bibitem{Persson} A.~Persson, \emph{Bounds for the discrete part of the spectrum of a semi-bounded Schrödinger operator}, Math. Scand. 8 (1960), 143–153, MR\href{https://doi.org/10.7146/math.scand.a-10602}{133586}

\bibitem{Ritore} M.~Ritor\'e, \emph{Isoperimetric inequalities in Riemannian manifolds}, Progr. Math., 348
Birkhäuser/Springer, Cham, 2023. xviii+460 pp, MR\href{https://doi.org/10.1007/978-3-031-37901-7}{4676392}

\bibitem{M Ritore} M.~Ritor\'e, \emph{Constant geodesic curvature curves and isoperimetric domains in rotationally symmetric surfaces}, Comm. Anal. Geom. 9 (2001), no. 5, 1093–1138, MR\href{https://doi.org/10.4310/CAG.2001.v9.n5.a5}{1883725}



\bibitem{Ros} A.~Ros, \emph{The isoperimetric problem}, Clay Math. Proc., 2
American Mathematical Society, Providence, RI, 2005
ISBN:0-8218-3587-4, MR{2167260}

\bibitem{Schmidt} E.~Schmidt, \emph{Beweis der isoperimetrischen Eigenschaft der Kugel im hyperbolischen und sphärischen Raum jeder Dimensionszahl}, Math. Z. 49 (1943), 1–109, MR\href{https://doi.org/10.1007/BF01174192}{0009127}

\bibitem{Talenti} G.~Talenti, \emph{The art of rearranging}, Milan J. Math. 84 (2016), no. 1, 105–157, MR\href{https://doi.org/10.1007/s00032-016-0253-6}{3503198}



\bibitem{Topping} P.~Topping, \emph{The isoperimetric inequality on a surface}, Manuscripta Math. 100 (1999), no. 1, 23–33, MR\href{https://doi.org/10.1007/s002290050193}{1714389}





\bibitem{Weil}A.~Weil,  \emph{Sur les surfaces a courbure negative},  CR Acad. Sci. Paris 182.2 (1926): 1069-71.

\bibitem{Yau} S.T.~Yau, \emph{Isoperimetric constants and the first eigenvalue of a compact Riemannian manifold}, Ann. Sci. École Norm. Sup. (4) 8 (1975), no. 4, 487–507, MR\href{http://www.numdam.org/item?id=ASENS_1975_4_8_4_487_0}{397619}




\end{thebibliography}
\end{document}